\documentclass[12pt,twoside]{article}
\usepackage{amsmath,amssymb,mathrsfs,graphicx,bbm,amsfonts,amsthm,mathtools}
\mathtoolsset{showonlyrefs}
\usepackage[utf8x]{inputenc}
\usepackage{nicefrac}
\usepackage{color}
\usepackage[normalem]{ulem}
\usepackage[english]{babel}

\usepackage{graphicx,subcaption}
\usepackage{url}
\usepackage{blindtext}

\usepackage[colorlinks=true, citecolor = green, linkcolor=blue, breaklinks, pdfauthor ={"Elena Magnanini, Giacomo Passuello"}]{hyperref}
\usepackage{todonotes}

\usepackage{comment}

\usepackage{graphicx}
\usepackage[title]{appendix}%

\newtheorem{theorem}{Theorem}[section]
\newtheorem{lemma}[theorem]{Lemma}
\newtheorem{proposition}[theorem]{Proposition}
\newtheorem{corollary}[theorem]{Corollary}
\newtheorem{definition}[theorem]{Definition}
\newtheorem{conjecture}[theorem]{Conjecture}
\newtheorem{remark}[theorem]{Remark}

\DeclareMathAlphabet{\mathbfit}{OML}{cmm}{b}{it}

\makeatletter
\@addtoreset{equation}{section}
\makeatother

\usepackage{titling}
    \pretitle{\centering\LARGE\scshape}
        \posttitle{\vskip 0.75cm}

    \predate{\vskip 0.75 cm \centering\large}
        \postdate{\par}

\newcommand{\ER}{Erd\H{o}s-R\'{e}nyi }

\newcommand{\cadlag}{\ifmmode\mathcal D\else c\`adl\`ag \fi}

\newcommand{\ba}{\begin{array}}
	\newcommand{\ea}{\end{array}}
\newcommand{\bea}{\begin{eqnarray}}
	\newcommand{\eea}{\end{eqnarray}}
\newcommand{\be}{\begin{equation}}
	\newcommand{\ee}{\end{equation}}

\def \R {{\mathbb R}}

\def \N {{\mathbb N}}
\def \P {{\mathbb P}}

\def \E {{\mathbb E}}

\def \cA {\mathcal{A}}

\def \cC {\mathcal{C}}

\def \cE {\mathcal{E}}

\def \cG {\mathcal{G}}

\def \cM {\mathcal{M}}
\def \cN {\mathcal{N}}

\def \cT {\mathcal{T}}

\def \a {{\alpha}}
\def \b {{\beta}}

\def \d {{\delta}}

\def \z {{\z}}

\def \z {{\zeta}}

\def \G {{\Gamma}}

\def \En {{\mathcal{E}_n}}

\def \aut {{\ell!}}

\def \1{\mathbbm{1}} 

\author{
Elena Magnanini \thanks{WIAS Berlin. \\ magnanini@wias-berlin.de} 
\and
Giacomo Passuello \thanks{Dipartimento di Matematica \textquotedblleft Tullio Levi-Civita",
Universit\`a di Padova.\\
giacomo.passuello@phd.unipd.it}
}
\title{Statistics for the triangle density in ergm  and its mean-field approximation}
\date{\today}

\usepackage[top=2.5cm, bottom=2.5cm,left=2.5cm,right=2.5cm]{geometry}

\begin{document}
\maketitle

\begin{abstract}
We consider the edge-triangle model (or Strauss model), and focus on the asymptotic behavior of the triangle density when the size of the graph increases to infinity.
This random graph belongs to the class of exponential random graphs, which follows the statistical mechanics approach of introducing a Hamiltonian to weigh the probability measure on the state space of graphs. In the analyticity region of the free energy, we prove a law of large numbers for the triangle density. Along the critical curve, where analyticity breaks down, we show that the triangle density concentrates with high probability in a neighborhood of its typical value.
A predominant part of our work is devoted to the study of a mean-field approximation of the edge-triangle model, where explicit computations are possible. In this setting we can go further, and additionally prove a standard and non-standard central limit theorem at the critical point, together with many concentration results obtained via large deviations and statistical mechanics techniques.
Despite a rigorous comparison between these two models is still lacking, we believe that they are asymptotically equivalent in many respects, therefore we formulate conjectures on the edge-triangle model, partially supported by simulations, based on the mean-field investigation.  
\end{abstract}
\newpage

\renewcommand{\baselinestretch}{0.9}\normalsize
\tableofcontents
\renewcommand{\baselinestretch}{1.0}\normalsize

\section{Introduction}
Exponential random graphs are an ubiquitous class of models able to capture
common network tendencies, such as clustering. Clustering is typically related to the presence of triangles in a graph. This can be easily motivated, for example, in the context of social networks; if two people share a common friend, or interest, it is more likely that they form a connection themselves. One of the main questions among sociologists is to understand how connectivity in local communities can affect the overall network structure \cite{YSB, GS, F1, F2}.   
This can be done by
introducing a probability measure that is  a function of the densities of certain given finite subgraphs (such as edges or triangles), thus biasing their occurrence, and then analyzing the large-scale properties of random networks sampled according to this distribution. Sampling is typically performed by means of Monte Carlo methods, such as the Glauber dynamics or the Metropolis–Hastings procedure \cite{BBS}. By following a statistical mechanics approach, the bias is encoded by a function called \emph{Hamiltonian}, contained in an exponential term, and the probability measure is then a \emph{Gibbs distribution} \cite{PN1, PN2, PN3}. Such class of graphs is particularly interesting since the Gibbs measure has the special property, among all other distributions (\emph{ensemble}) over the space of simple graphs, to assign higher probability to graphs that fit better some prescribed  constraints imposed by a given set of observations. This is of course very helpful if we want to build a mathematical model that encodes properties that are similar to those observed in a real network. Specifically, according to the Maximum Entropy
Principle introduced by Jaynes \cite{J1,J2}, this distribution maximizes the so-called \emph{Gibbs entropy} subject to the known constraints.

From a probabilistic perspective, exponential random graphs represent a generalization of the dense \ER random graph \cite{ER}, as their probability distribution is obtained by tilting the \ER measure by an exponential weight that contains different subgraph densities, thus introducing some dependence between the random edges. 
Many important and rigorous results on the model have been obtained so far \cite{CD, CDey, ChNotes}, sometimes imposing constraints on subgraph densities \cite{KY,NRS}. 

Our analysis will be focused on the (unconstrained) edge-triangle model \cite{Str86}, a two-parameter family of exponential random graphs in which dependence between the random edges is defined through triangles, and its \emph{mean-field} approximation \cite{BCM}.
In both cases, a crucial point, which is also preparatory for understanding phase transitions, is the derivation of the limiting \emph{free energy}. For the edge-triangle model, the analytical expression of this function is known, together with its phase diagram, in a region of parameters called \emph{replica symmetric regime} \cite{CD, CDey},
where it can be characterized as the solution of a one-dimensional maximization problem. Inside the replica symmetric region the analyticity of the free-energy may break down, thus revealing a phase transition \cite{CDey, RY, AZ}. This happens along a critical curve that ends in a critical point where the phase transition is of second order. As expected, the edge-triangle model and its mean-field approximation share the same limiting free energy, and this is a crucial property that motivates all our analysis.  

Regarding the first model, one of the key results of our paper is the characterization of the limiting behavior of the triangle density (as the graph size $n$ tends to infinity) within the replica-symmetric regime. Our analysis provides a strong law of large numbers whenever the parameters $(\alpha,h)$ (of triangles and edges, respectively) are taken outside the critical curve, and proves that on the critical curve the triangle density concentrates with high probability in a neighborhood of the free energy maximizers. Our results extend to other graph statistics. 

As mentioned, the predominant part of our investigation is concerned with the mean-field approximation. Its major advantage is that the Hamiltonian can be expressed as a function of the edge density, and exact computations are possible (like in the \emph{Curie--Weiss} model). In this setting, we consider an approximated  triangle density for which we prove  a central limit theorem whenever the parameters lie outside the critical curve and away from the critical point. We also characterize the fluctuations of the triangle density at the critical point, presenting a non-standard central limit theorem with scaling exponent  $3/2$.
Some heuristic computations based on large deviations estimates, suggest that the triangle density in the edge-triangle model may exhibit the same limiting behavior as the mean-field approximation. We then formulate conjectures about fluctuations of the triangle density in the whole replica symmetric regime (which also includes the critical curve and the critical point).

The paper is organized as follows. In Sec.~\ref{model} we introduce the exponential random graph family and we recall many results present in the literature. We then focus on the edge-triangle model and its mean-field approximation and collect some main properties.  In Sec.~\ref{results} we state our results, first for the edge-triangle model, and then for its mean-field approximation. Secs.~\ref{proofs_et}--\ref{proofs-mf} are devoted to the proofs. Sec.~\ref{proofs_et} contains the proof of the strong law of large numbers given in Thm.~\ref{Thm_LLN}, which is based on the exponential convergence of the sequence of triangle densities, together with the proof of Thm.~\ref{Thm_convergence_in_distribution}, which provides a
concentration result valid on the critical curve. Thms.~\ref{Thm_LLN_Hk}--\ref{genHk} generalize these statements to any simple graph. Sec.~\ref{proofs-mf} is entirely devoted to the mean-field model: we prove the analogs of the results
derived for the edge-triangle model in the previous section, 
sometimes in a stronger form, and we can go further, by characterizing the fluctuations of the triangle density up to the critical point (Thms.~\ref{Thm_standard_CLT} and \ref{Thm_non-standard_CLT}).  
These results are extended in Thm.~\ref{Thm_conditional_limit_theorems_mfm} to the critical curve under a suitable conditional measure. 
As a byproduct of this analysis, in Props.~\ref{prop_speed_mmf} and \ref{prop_speed_conditional_mfm}, we determine the speed of convergence of the average triangle density in the unconditional and conditional settings. Finally, Sec.~\ref{CLT_discussion} is dedicated to a less rigorous discussion, where we present the conjectures on the edge-triangle model inspired by the analysis of its mean-field approximation, supported by both heuristic computations and simulations. 

\section{Models and background}\label{model}
\subsection{Exponential random graphs} \label{first_subsect}
As we pointed out before, exponential random graphs (ERGs) are a class of random graphs where the probability measure over the state space is tilted by a function called Hamiltonian, devised to enhance or decrease the probability of certain structures in the graph. Before giving its definition we need to introduce the notion of homomorphism density. 

\begin{definition}[homomorphism]\label{hom}
A homomorphism from a graph $H$ to a graph $G$ is a mapping (not necessarily bijective)  $\varphi: V(H)\rightarrow V(G)$ such that if
$u, w$ are adjacent in $H$, then   $\varphi(u), \varphi(w)$ are adjacent in $G$.
\end{definition}
Note that $\varphi$ preserves adjacency but not non-adjacency. We denote by $|\text{hom}(H,G)|$ the number of homomorphism of a graph $H$ in $G$. For example, if $H$ is a triangle then  $|\text{hom}(H,G)|= 6 T_n(G)$, where $T_n(G)$ denotes the number of triangles in $G$. If $H$ is not a complete subgraph, the counting is more complicated. For example, if $H$ is a \emph{two-star} \footnote{A two-star is a path of length $2$, which has $3$ vertices and $2$ edges.} (or \emph{wedge}) and $G$ is a triangle, then $|\text{hom}(H,G)|= 3\cdot 2^2$; indeed, there are three copies of $H$ in $G$ (one for each root) and for each of them 4 possible homorphisms.
We define the homomorphism density as
\begin{equation}\label{def_graph_hom_density}
t(H,G) := \frac{|\text{hom}(H,G)|}{|V(G)|^{|V(H)|}}.
\end{equation}

In this paper we deal with simple graphs on $n$ labeled vertices and we denote by  $\mathcal{G}_n$ the set of such graphs. 
For any $k \in \mathbb{N}$, we consider $H_1, H_2, \dots, H_k$ \footnote{Here the subscript denotes the progressive numbering and has nothing to do with the number of vertices or edges. } pre-chosen finite simple graphs (such as edges, stars, triangles, cycles, \dots) weighted by a collection of real parameters contained in the vector $\boldsymbol{\beta}=(\beta_1, \dots, \beta_k)$. The Hamiltonian is a function $\mathcal{H}_{n;\boldsymbol{\beta}}: \mathcal{G}_n \to \mathbb{R}$ defined as 
\begin{equation}\label{Hamiltonian}
\mathcal{H}_{n;\boldsymbol{\b}}(G):=n^2\sum_{i=1}^{k}\beta_{i}t(H_{i},G)\,,
\quad \mbox{ for } G\in\cG_n.
\end{equation}
As probability measure on the space $\mathcal{G}_n$ we take the Gibbs probability density
\be\label{Gibbs}
\mu_{n; \boldsymbol{\b}} (G):=\frac{\exp \left(\mathcal{H}_{n,\boldsymbol{\b}}(G)\right)} {Z_{n;\boldsymbol{\b}}}, \quad \text{with } Z_{n;\boldsymbol{\b}}:=\sum_{G \in \cG_{n}} \exp \left(\mathcal{H}_{n;\boldsymbol{\b}}(G)\right),
\ee
where the normalizing constant $Z_{n;\boldsymbol{\beta}}$ is called {\em partition function}. Random graphs whose distribution is a Gibbs measure of the form \eqref{Gibbs} are called exponential random graphs. We will denote the related Gibbs measure and average
by $\P_{n;\boldsymbol{\b}}$ and $\E_{n;\boldsymbol{\b}}$, respectively.
Two crucial functions for studying the model are the finite-size and infinite-size free energy:
\be\label{FreeEnergy}
f_{n;\boldsymbol{\b}}\,:= \frac{1}{n^2}\ln Z_{n;\boldsymbol{\b}}\,\quad \text{ and } \quad
f_{\boldsymbol{\b}}\,:=\lim_{n\to+\infty} f_{n;\boldsymbol{\b}}
\,.\ee
A lack of analyticity in $f_{\boldsymbol{\b}}$ characterizes the presence of a phase transition.
An explicit expression of this function has been obtained in \cite{CD} when the vector of parameters $\boldsymbol{\beta}$ lies in a specific region called \emph{replica-symmetric regime} (term borrowed from spin glasses theory). As stated in \cite[Thm.~4.1]{CD}, if $\beta_{2},\dots, \beta_{k}$ are non-negative, then \begin{equation}\label{scalar_probl}
			f_{\boldsymbol{\b}} \,
			=\,
			\sup_{0\leq u\leq\,1}\left(\sum_{i=1}^{k}\beta_i\,u^{E(H_i)} -\frac{1}{2}I(u)\right),
		\end{equation}
		where $E(H_i)$ denotes the number of edges in $H_i$ and $I(u):=u\ln u + (1-u)\ln(1-u)$.
Despite this result covers only non-negative values of the parameters, the replica symmetric regime can be slightly extended including (not too big) negative values of $\beta_2, \dots, \beta_k$ (see \cite{CD}, Thm.~4.2). More precisely, \eqref{scalar_probl} holds whenever $\beta_2, \dots, \beta_k$ are such that 
\begin{equation}\label{beta_piccoli_rs}
\sum_{i=2}^{k} |\beta_{i}|E(H_{i})(E(H_{i})-1)<2 \,.
\end{equation}

\subsection{Presentation of the models}
The \emph{edge-triangle} or \emph{Strauss model}~\cite{Str86} is obtained by considering only the contribution of edges and triangles in the Hamiltonian \eqref{Hamiltonian}. By convention we  assume $H_1$ to be a single edge and $H_2$ to be a triangle. More precisely, by setting $\beta_3=\cdots=\beta_k=0$ in \eqref{Hamiltonian}, we get
\[
\mathcal{H}_{n;\boldsymbol{\beta}} (G) = n^2 \left[ \beta_1 t(H_1,G) + \beta_2 t(H_2,G)\right] \qquad G\in\mathcal{G}_n\,.
\]
Let $E_n(G)$ (resp.~$T_n(G)$) denote  the number of edges (resp.~triangles) in $G$. By recalling Def.~\ref{hom} of homomorphism density, we have
\begin{equation}\label{ET_hom}
t(H_1,G) = \frac{2E_n(G)}{n^2} \quad \text{ and } \quad t(H_2,G) = \frac{6T_n(G)}{n^3}.
\end{equation}
Therefore, by performing the change of variable $h:=2\beta_1$; $\alpha:=6\beta_2$, we can equivalently consider
\be\label{Hamiltonian-ETmodel}
\mathcal{H}_{n; \a,h}(G) = \frac{ \a}{ n}  T_n(G) +  h  E_n(G) \quad \text{ with } \a,h\in \R\,.
\ee
We will denote by $\P_{n;\a,h}$ the Gibbs measure related to this Hamiltonian, and by $\E_{n; \a,h}$ the corresponding expectation. Notice that in this setting condition \eqref{beta_piccoli_rs} reads $|\alpha|= 6 |\beta_2|<2$, and, therefore, the replica symmetric regime coincides with the region $\alpha>-2$, $h\in\mathbb{R}$. The free energy \eqref{scalar_probl} reduces then to 
\begin{equation}\label{free_energy}
f_{\a,h} \, = \,  \sup_{0 \leq u \leq 1} \left(\frac{\alpha}{6}u^3 + \frac{h}{2}u - \frac{1}{2}I(u) \right) \, = \, \frac{\a}{6}{(u^{*})}^3 +\frac{h}{2}u^{*} -\frac{1}{2}I(u^{*}),
\end{equation}
where $I(u)$ is defined below \eqref{scalar_probl} and $u^{*}=u^*(\alpha,h)$ is a maximizer that solves the fixed-point equation
\begin{equation}\label{FixPointEq}
\frac{e^{\alpha\,u^{2} +h}}{1+e^{\alpha\,u^{2} +h}}=u\,.
\end{equation}
A numerical investigation of the optimizers of the free energy when $\alpha$ is negative and $|\alpha|$ is large has been done in \cite{GGM}. Equation \eqref{FixPointEq} can admit more than one solution at which the supremum in \eqref{free_energy} is attained, and this denotes the presence of a phase transition inside the replica symmetric regime. When the parameters $\alpha ,h$ are chosen in this region, the edge-triangle model, when $n$ goes to infinity, becomes \textit{indistinguishable} from an \ER \textit{graphon} with connection probability $u^*$ (we refer the reader to Sec.~\ref{graphs_limit}, where these notions are made precise). This remains true even when the supremum is not unique; in this case the parameter $u^*$ is randomly chosen according to some (unknown) probability distribution on the set of solutions of \eqref{free_energy} (see \cite{CD}, Thm.~4.2). The effect of the phase transition is then a jump between very different values of the limiting \ER graphon $u^*$.

\paragraph{Phase diagram.}
We recall that the limiting free energy $f_{\alpha,h}$ is well defined on the whole replica symmetric regime $\alpha>-2$, $h\in\mathbb{R}$. However, the fixed point equation \eqref{FixPointEq} can admit more than one solution, and this is strictly related to the loss of analyticity of $f_{\alpha, h}$. More precisely, \eqref{FixPointEq} has exactly one solution on the whole replica symmetric regime except for a certain critical curve $\mathcal{M}^{rs}$ that starts at the critical point $(\alpha_c,h_c) := \left(\frac{27}{8},\ln 2 -\frac{3}{2}\right)$  and that can be written as $h=q(\alpha)$ for a (non-explicit) continuous and strictly decreasing function $q$:
\begin{equation}\label{curve_rs}
\mathcal{M}^{rs} := \left\{(\alpha,h) \in (\alpha_c,+\infty) \times (-\infty,h_c): h = q(\alpha)\right\}.
\end{equation}
\begin{figure}[h!] 
\centering
\includegraphics[scale=0.7]{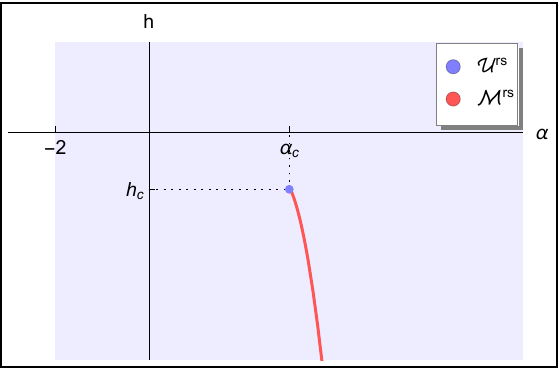}
\caption{\small Illustration of the phase in replica symmetric regime taken from \cite{BCM}. The  curve $\mathcal{M}^{rs}$ \eqref{curve_rs} represents the region of $(\alpha,h)$ where the optimization problem~\eqref{free_energy} admits two solutions. Inside the blue region, that includes the critical point $(\alpha_c,h_c)$, the scalar problem \eqref{free_energy} admits a unique solution.
}\label{fig:phase_diagram}
\end{figure}
In \cite[Prop.~3.2]{RY} it has been proved that off the curve $\mathcal{M}^{rs}$ the scalar problem \eqref{free_energy} admits one solution, whereas it has exactly two solutions along the curve $\mathcal{M}^{rs}$. The free energy is analytic on the region $\mathcal{U}^{rs}\setminus \{(\alpha_c,h_c)\}$, where
\begin{equation}\label{reg_rs}
\mathcal{U}^{rs} := \left( (-2,+\infty) \, \times \mathbb{R} \right) \setminus \mathcal{M}^{rs}.
\end{equation}
 Moreover, at the critical point $(\alpha_c,h_c)$ the second order partial derivatives of $f_{\alpha,h}$ diverge (see \cite{RY}, Thm.~2.1), therefore we observe a second order phase transition. Finally, along the curve $\mathcal{M}^{rs}$ the first order partial derivatives of $f_{\alpha,h}$ have jump discontinuities, and we observe a first order phase transition. 
Fig.~\ref{fig:phase_diagram} provides a qualitative representation of the phase diagram.

\paragraph{Approximating the number of triangles.}
 As a consequence of the convergence of the ERG to the \ER graphon with parameter $u^*$ (which holds in probability w.r.t. the so-called  \emph{cut distance}, see \cite[Thm.~4.2]{CD}), we can heuristically approximate the triangle density (as well as other graph-statistics) in the large $n$ limit. The \ER random graph with parameter $u^*$ and $n$ vertices has, on average, ${u^*}^3 \binom{n}{3}$ triangles and $u^*\binom{n}{2}$ edges. We observe that 
$${u^*}^3 \binom{n}{3} \approx \frac{4}{3n^3}\left(u^* \binom{n}{2}\right)^3.$$
What we expect is that the same holds, within the replica symmetric regime and when $n$ is large, for the ERG. Thus, we introduce the approximated count of triangles
\begin{equation} \label{barT}
T_n(G)\approx \frac{4}{3n^3}E_n(G)^3=: \bar{T}_n(G). 
\end{equation}
Alternatively, going back to Def.~\ref{def_graph_hom_density} of homomorphism density, we can equivalently say that we approximate the number of triangles $T_n(G)=\frac{n^{3}t(H_2,G)}{6}$ (see \eqref{ET_hom}) with $\bar{T}_n(G)=\frac{n^{3}t^3(H_1,G)}{6}$.

\paragraph{Mean-field approximation.}
Definition \eqref{barT} leads to the following mean-field approximation, originally introduced in \cite{BCM}, of the edge-triangle Hamiltonian \eqref{Hamiltonian-ETmodel}:
\begin{equation}\label{H_mf}
	\bar{\mathcal{H}}_{n;\alpha,h}(G):=\frac{\alpha}{n}\bar{T}_n(G)+hE_n(G)\,, \quad \text{ for } G\in\mathcal{G}_n.
\end{equation}
We borrow this terminology from statistical mechanics, due to the similarities with the \emph{Curie--Weiss} model (see e.g.~\cite[Chap.~2]{FV}), which we are going to highlight further in the next paragraph. The big advantage of Hamiltonian \eqref{H_mf} is that it is just a function of the one dimensional
parameter $t(H_1, G) = \frac{2 E_n(G)}{n^2}$,
taking values in $\G_{n} :=\left\{0,\frac{2}{n^2}, \dots,1-\frac{1}{n}\right\}$.
We denote by $\bar{\P}_{n;\a,h}$ and $\bar{\E}_{n;\a,h}$ the corresponding measure and expectation, respectively. Moreover, as usual, we define the finite size free energy as
\begin{equation}\label{fe_mf}
	\bar{f}_{n;\alpha,h}:= \frac{1}{n^{2}}\ln \bar{Z}_{n;\alpha,h}.
\end{equation}
A crucial property of this approximated model is the following (see \cite[Thm.~8.2]{BCM}). Let $(\alpha,h) \in (-2,+\infty) \times \mathbb{R}$ and let $f_{\alpha,h}$ as in \eqref{free_energy}. Then
	\begin{equation}
		\label{equiv_fren}
		\lim_{n\to+\infty}\bar{f}_{n;\alpha,h}= f_{\alpha,h}\,.
	\end{equation}
In other words, the edge-triangle model and this mean-field approximation share the same infinite volume free energy; this result will be extensively used in the proofs of Sec. \ref{proofs-mf}.

\subsection{Notation and preliminaries}

We denote by $\mathcal{E}_n$ the edge set of the complete graph
on $n$ vertices, with elements  labeled from 1 to $\binom{n}{2}$ and we set $\mathcal{A}_n := \{0,1\}^{\mathcal{E}_n}$. We observe that there is a one-to-one correspondence between $\cA_n$ and the set of $n\times n$ symmetric adjacency matrices with zeros on the diagonal and the graphs in $\cG_n$. As a consequence, to each graph $G\in\cG_n$ we can associate an element ${x}=(x_{i})_{i \in \cE_n}\in\cA_n$ where $x_{i}=1$ if the edge $i$ is present in $G$, and $x_{i}=0$ otherwise.
With an abuse of nomenclature, in the rest of the paper we will refer to the elements of $\mathcal{A}_n$ as adjacency matrices, and we will write $E_n(x)=E_n(G)$, $T_n(x)=T_n(G)$ and $\bar{T}_n(x)=\bar{T}_n(G)$ whenever $x \in \cA_n$ is the adjacency matrix of a graph $G \in \cG_n$. This representation allows for the following equivalent formulation of the Hamiltonians \eqref{Hamiltonian-ETmodel}--\eqref{H_mf}, as  functions on $\cA_n$:
\begin{equation}\label{Hamilt_ERG}
\mathcal{H}_{n;\alpha,h}(x) = \frac{\alpha}{n} \sum_{\{i,j,k\} \in \mathcal{T}_n} x_i x_j x_k + h \sum_{i \in \mathcal{E}_n} x_i,
\end{equation}
\begin{equation}\label{Hamilt_MF}
\bar{\mathcal{H}}_{n;\alpha,h}(x) = \frac{4\alpha}{3n^4}\Big(\sum_{i \in \mathcal{E}_n} x_i\Big)^3  + h \sum_{i \in \mathcal{E}_n} x_i,
\end{equation}
where $\mathcal{T}_n:=\{\{i,j,k\} \subset \mathcal{E}_n: \{i,j,k\} \text{ is a triangle}\}$. 
The Gibbs probability $\mathbb{P}_{n;\a,h}$ (resp.~$\bar{\mathbb{P}}_{n;\a,h}$) will act consequently on $\cA_n$.
\begin{remark}\label{consist_cond}
The sequence of measures $(\mathbb{P}_{n;\alpha,h})_{n \geq 1}$ (as well as $(\bar{\mathbb{P}}_{n;\alpha,h})_{n \geq 1}$) satisfies proper consistency conditions allowing for the application of Kolmogorov Existence Theorem (see, for example, Appendix A.7 in \cite{E}). 
As a consequence, there exists a unique probability measure $\P_{\a,h}$ on the space $\left(\{0,1\}^{ \N},\mathcal{B}(\{0,1\}^{\N})\right)$
with marginals corresponding to the measures $\P_{n;\a,h}$, for all $n\in\N$ (here $\mathcal{B}$ denotes the Borel  $\sigma-$algebra).
\end{remark}
\begin{remark}
Note that Hamiltonian \eqref{Hamilt_ERG} has the same form of the energy function typically used in interacting particle systems. By making a parallelism with the Curie--Weiss model,  we can think of an ERG as a system where each edge is a particle having a \emph{spin} ($0$ or $1$), which interacts with its neighbors. The notion of \textquotedblleft neighbor" depends on the specific choice of the subgraphs $H_1,\dots,H_k$; for the edge-triangle model, two edges are neighbors if they are adjacent.
This interaction is local, however, if we ignore the relative position of edges, we recover \eqref{Hamilt_MF}: 
$$\sum_{\{i,j,k\} \in \mathcal{T}_n} x_i x_j x_k 
= \sum_{i \in \En} x_i \sum_{\substack{j,k \in \mathcal{E}_n:\\ \{i,j,k\} \in \mathcal{T}_n}} x_j x_k 
\approx  \sum_{i \in \En} x_i \sum_{j,k \in \mathcal{E}_n} \frac{4x_j x_k}{3n^3}\,,$$
where the factor $4$ appears when we replace the number of wedges roughly with $\left(\frac{2}{n^2}\sum_{i \in \En} x_i \right)^2$, and adjusting the normalization in accordance with the choice $x=(x_i)_{i\in \En}\equiv 1$. The factor $1/3$ avoids overcounting.
\end{remark}
We are interested in understanding the asymptotic behavior of the number of triangles. Let $X = (X_i)_{i\in\cE_n} \in \cA_n$ be the random adjacency matrix of an ERG with law $\mu_{n;\a,h}$. We consider the random variables
\be\label{def:E_n}
E_n \, \equiv \, E_n(X) = \, \sum_{i\in\cE_n} X_i \, ,
\ee
\be\label{def:T_n}
T_n \, \equiv \, T_n(X) = \, \sum_{\{i,j,k\}\in\cT_n} X_iX_jX_k \, ,
\ee
\be\label{def:bT_n}
\bar{T}_n \, \equiv \, \bar{T}_n(X) = \, \frac{4}{3 n^3} \Big( \sum_{i\in\cE_n} X_i\Big)^3.
\ee

We prove classical limit theorems of the sequences $(T_n)_{n\geq 1}$ and $(\bar{T}_n)_{n\geq 1}$ inside the replica symmetric regime. A predominant part of our results is concerned with the sequence $(\bar{T}_n)_{n\geq 1}$, since the mean-field approximation encoded by the Hamiltonian \eqref{Hamilt_MF} allows for explicit computations.

\begin{definition}
For each $n \in \mathbb{N}$, we define the \textbf{average} and the \textbf{variance of the triangle density}, respectively of the edge-triangle model and of the mean-field approximation, as
	\begin{align}\label{average&variance_edge_density}
		m^{\Delta}_n(\a,h):=\frac{6\E_{n;\a,h}\left(\frac{T_n}{n}\right)}{n^2} \quad \text{ and } \quad v^{\Delta}_{n}(\a,h):= \partial_\a m^{\Delta}_n(\a,h)\\
		\bar{m}^{\Delta}_n(\a,h):=\frac{6\bar{\E}_{n;\a,h}\left(\frac{\bar{T}_n}{n}\right)}{n^2} \quad \text{ and } \quad \bar{v}^{\Delta}_{n}(\a,h):= \partial_\a \bar{m}^{\Delta}_n(\a,h). \label{average&variance_edge_density_mf}
	\end{align}

\end{definition}
It is easy to see that 
\begin{align}
\frac{\E_{n;\a,h}\left(T_n\right)}{n^3}= \partial_{\a}f_{n;\a,h} \quad &\text{ and } \quad  \frac{\text{Var}_{n;\a,h}\left(T_n\right)}{n^3}= \partial_{\a\a}f_{n;\a,h}\label{var_rel}\\
\frac{\bar{\E}_{n;\a,h}\left(\bar{T}_n\right)}{n^3}= \partial_{\a}\bar{f}_{n;\a,h} \quad &\text{ and } \quad  \frac{\text{Var}_{n;\a,h}\left(\bar{T}_n\right)}{n^3}= \partial_{\a\a}\bar{f}_{n;\a,h}, \notag
\end{align}
therefore, $m^{\Delta}_n(\a,h)= 6\partial_{\a}f_{n;\a,h}$ and $v^{\Delta}_{n}(\a,h)=6\partial_{\a\a}f_{n;\a,h}$ (and the same holds for $\bar{m}^{\Delta}_n(\a,h)$ and $\bar{v}^{\Delta}_{n}(\a,h)$, replacing  $f_{n;\a,h}$ with $\bar{f}_{n;\a,h}$).
In the rest of the paper we will use the following notation to distinguish the optimizer(s) of the scalar problem \eqref{free_energy}, sometimes dropping the dependence on $(\alpha, h)$ to the sake of readability: 
 
\begin{equation}\label{u^*}
\begin{cases}
		u_0^*(\a,h) &\text{ if } (\a,h)\in \mathcal{U}^{rs}\setminus \{(\a_c,h_c)\}, \\
		u_1^*(\a,h) \text{ and } u_2^*(\a,h) &\text{ if } (\a,h) \in \mathcal{M}^{rs}, \\
		u_c^*(\alpha,h)=\frac{2}{3} &\text{ if } (\a,h)=(\a_c,h_c)\,.
	\end{cases}
\end{equation}
We are now ready for stating our results.

\section{Main results}\label{results}

\subsection{Edge-triangle model}\label{RET}

	\begin{theorem}[SLLN for $T_n$]\label{Thm_LLN}
	For all $(\alpha,h)\in\mathcal{U}^{rs}$,
	\begin{equation}\label{LLN}
		\frac{6T_n}{n^3} \, \xrightarrow{\;\;\mathrm{a.s.}\;\;}{} \, {u^{*}_0}^{3}(\alpha,h) \qquad \text{w.r.t. } \mathbb{P}_{\a,h}, \text{ as } n \to +\infty,
	\end{equation}	
	where $u^*_0$ solves the maximization problem \eqref{free_energy}.
\end{theorem}

	\begin{theorem}\label{Thm_convergence_in_distribution}
	For all $(\alpha,h) \in \mathcal{M}^{rs}$ and for all sufficiently small $\varepsilon>0$, there exists a constant $\tau=\tau(\varepsilon;\alpha,h)>0$ such that if
	$$J(\varepsilon):= ({u_1^{*}}^{3}(\alpha,h)-\varepsilon,{u_1^{*}}^{3}(\alpha,h)+\varepsilon) \cup  ({u_2^{*}}^3(\alpha,h)-\varepsilon,{u_2^{*}}^3(\alpha,h)+\varepsilon),$$
	then, for large enough $n$
	\[
	\mathbb{P}_{n;\alpha,h}\left(\frac{6T_n}{n^3}\in J(\varepsilon)\right)\geq 1-e^{-\tau n^{2}}, 
	\]
	where $u_1^*(\alpha,h)$ and $u_2^*(\alpha,h)$ are the two maximizers of the scalar problem \eqref{free_energy}.
\end{theorem}

By replacing the triangle homomorphism density by the homomorphism density of any simple graph, we immediately obtain the following generalizations of  Thm.~\ref{Thm_LLN} and Thm.~\ref{Thm_convergence_in_distribution}.

\paragraph{Generalization to a generic simple graph.}
Fix $k\in\mathbb{N}$ with $k>2$. Let $H_k$, be a pre-chosen finite simple graph (such as a  square, a cycle, a clique \dots) with $E(H_k)$ edges, and let $t(H_k, G)$ be the homomorphism density \eqref{def_graph_hom_density} of $H_k$.
The following theorems characterize the asymptotic behavior of $(t(H_k, \cdot))_{n\geq 1}$\footnote{This notation implies that we are considering the sequence of random variables.}  in the replica symmetric regime. 

	\begin{theorem}\label{Thm_LLN_Hk}
	For all $(\alpha,h)\in\mathcal{U}^{rs}$,
	\begin{equation}\label{LLN_Hk}
		t(H_k, \cdot) \, \xrightarrow{\;\;\mathrm{a.s.}\;\;}{} \, {u^{*}_0}^{E(H_k)}(\alpha,h) \qquad \text{w.r.t. } \mathbb{P}_{\a,h}, \text{ as } n \to +\infty,
	\end{equation}	
	where $u^*_0$ solves the maximization problem \eqref{free_energy}.
\end{theorem}

	\begin{theorem}\label{genHk}
	For all $(\alpha,h) \in \mathcal{M}^{rs}$ and for all sufficiently small $\varepsilon>0$, there exists a constant $\tau=\tau(\varepsilon;\alpha,h)>0$ such that if
	$$J(\varepsilon):= ({u_1^{*}}^{E(H_k)}(\alpha,h)-\varepsilon,{u_1^{*}}^{E(H_k)}(\alpha,h)+\varepsilon) \cup  ({u_2^{*}}^{E(H_k)}(\alpha,h)-\varepsilon,{u_2^{*}}^{E(H_k)}(\alpha,h)+\varepsilon),$$
	then, for large enough $n$
	\[
	\mathbb{P}_{n;\alpha,h}\left(t(H_k, \cdot)\in J(\varepsilon)\right)\geq 1-e^{-\tau n^{2}}, 
	\]
	where $u_1^*(\alpha,h)$ and $u_2^*(\alpha,h)$ are the two maximizers of the scalar problem \eqref{free_energy}.
\end{theorem}

After having proved a SLLN, it would be natural to investigate the fluctuations of the triangle density around its mean value. In Sec.~\ref{CLT_discussion} we perform simulations and provide conjectures, also based on the mean-field investigation of Sec.~\ref{proofs-mf}. It is in order to stress that the Yang-Lee  theorem \cite[Thm.~2]{LY}, which is a powerful tool when it comes to prove a central limit theorem, is not applicable to our case, since $Z_{n;\alpha,h}$ does not admit a polynomial representation in $z:=e^{\alpha}$.

\subsection{Mean-field approximation} \label{mean-field-results}
	\begin{theorem}[SLLN for $\bar{T}_n$]\label{Thm_LLN_mf}
	For all $(\alpha,h)\in\mathcal{U}^{rs}$, 
	\[
	\frac{6\bar{T}_n}{n^3} \, \xrightarrow{\;\;\mathrm{a.s.}\;\;}{} \, {u^{*}_0}^3(\alpha,h) \quad \text{ w.r.t. } \bar{\mathbb{P}}_{\alpha,h}, \text{ as } n \to +\infty,
	\]
	where $u^*_0$ solves the maximization problem in \eqref{free_energy}.
\end{theorem}

\begin{theorem}\label{Thm_convergence_in_distribution_mf}
	For all $(\alpha,h) \in \mathcal{M}^{rs}$,
	\[
	\frac{6\bar{T}_n}{n^3} \, \xrightarrow{\;\;\mathrm{d}\;\;}{} \, \kappa \delta_{{u_1^{*}}^3(\alpha,h)}+
	(1-\kappa)\delta_{{u_2^{*}}^3(\alpha,h)}  \quad \text{ w.r.t. } \bar{\mathbb{P}}_{n;\alpha,h}, \text{ as } n \to +\infty,
	\]
	where $u_1^*$, $u_2^*$ solve the maximization problem in \eqref{free_energy}, and
\begin{equation}\label{kappa}	
\kappa:=\frac{
		\sqrt{\left[{1-2\alpha \left( u_1^*\right)^2(1-u_1^*)}\right]^{-1}}
	}
	{
		\sqrt{\left[{1-2\alpha \left( u_1^*\right)^2(1-u_1^*)}\right]^{-1}}
		+
		\sqrt{\left[{1-2\alpha \left( u_2^*\right)^2(1-u_2^*)}\right]^{-1}}
	} \,.
\end{equation}
\end{theorem}

	\begin{theorem}[CLT for $\bar{T}_n$]\label{Thm_standard_CLT}
	If $(\a,h) \in \mathcal{U}^{rs}\setminus \{(\a_c,h_c)\}$, 
	\[	
	\sqrt{6} \, \frac{\frac{\bar{T}_n}{n} - \frac{n^{2}}{6}\bar m^{\Delta}_n(\a,h)}{n} \xrightarrow{\;\;\mathrm{d}\;\;}{} \mathcal{N}(0,\bar{v}_0^\Delta(\alpha,h)) \quad \text{ w.r.t. } \bar{\mathbb{P}}_{n;\alpha,h},\text{ as } n \to +\infty,
	\]
	where $\mathcal{N}(0,\bar{v}_0^\Delta(\alpha,h))$ is a centered Gaussian distribution with
	variance 	
	\begin{equation} \label{varCLT}
	\bar{v}_0^\Delta(\alpha,h) :=  \frac{3{u^*_0}^4(\alpha,h)}{4 c_0},
	\end{equation} 
being $c_0\equiv c_0(\alpha,h):= \frac{1-2\alpha[u^*_0(\alpha,h)]^2[1-u^*_0(\alpha,h)]}{4u^*_0(\alpha,h)[1-u^*_0(\alpha,h)]}$.
\end{theorem}

\begin{theorem}
	[Non-standard CLT for $\bar{T}_n$]\label{Thm_non-standard_CLT}
	If $(\a,h)=(\alpha_c,h_c)$,
	\[	
	6 \, \frac{\frac{\bar{T}_n}{n} - \frac{n^{2}}{6}\bar m^{\Delta}_n(\a_c,h_c)}{n^{3/2}} \xrightarrow{\;\;\mathrm{d}\;\;}{} \bar Y \quad \text{ w.r.t. } \bar{\mathbb{P}}_{n;\alpha_c,h_c},\text{ as } n \to +\infty,
	\]
	where $\bar Y$ is a generalized Gaussian random variable with Lebesgue density $\bar{\ell}^c(y)\propto e^{-\frac{3^8}{2^{14}}y^{4}}$.	
\end{theorem}

\begin{proposition}\label{prop_speed_mmf}
	For all $(\alpha,h) \in \mathcal{U}^{rs}\setminus \{(\a_c,h_c)\}$,
	\begin{equation}\label{speed_mmf1}
		\lim_{n\to+\infty}
		n \cdot\bar{\mathbb{E}}_{n;\alpha,h}\left( \left\vert \frac{6\bar{T}_n}{n^3} - {u_0^*}^3(\a,h) \right\vert\right)
		=\E(|\bar{X}|)\,,
	\end{equation}
	where $\bar{X}$ is a centered Gaussian random variable with
	variance $6\bar{v}^{\Delta}_0 (\alpha,h)=\frac{9{u^*_0}^4(\alpha,h)}{2 c_0}$, \newline 
	being $c_0\equiv c_0(\alpha,h)= \frac{1-2\alpha[u^*_0(\alpha,h)]^2[1-u^*_0(\alpha,h)]}{4u^*_0(\alpha,h)[1-u^*_0(\alpha,h)]}>0$.
	Moreover, at the critical point
	\begin{equation}\label{speed_mmf2}
		\lim_{n\to+\infty}
		\sqrt{n}\cdot\bar{\mathbb{E}}_{n;\alpha_c,h_c}
		\left( \left\vert \frac{6\bar{T}_n}{n^3} - {u^*}^3(\a_c,h_c) \right\vert\right)
		=\E(|\bar Y|)\,,
	\end{equation}
	where $\bar Y$ is a generalized Gaussian random variable with Lebesgue density
	$\bar\ell^c(y)\propto e^{-\frac{3^8}{2^{14}}y^4}$.
\end{proposition}

\begin{corollary}\label{cor_speed_mf}
	For all $(\alpha,h) \in \mathcal{U}^{rs}\setminus \{(\a_c,h_c)\}$, we have
	\begin{equation}\label{speed_mf1}
		\lim_{n\to+\infty} n \cdot (\bar m^{\Delta}_n(\a,h)-{u_0^*}^3(\a,h)) =0\,,
	\end{equation}
	while for $(\a,h)=(\a_c,h_c)$ 
	\begin{equation}\label{speed_mf2}
		\lim_{n\to+\infty} \sqrt{n}\cdot(\bar m^{\Delta}_n(\a_c,h_c)-{u^*}^3(\a_c,h_c)) =0\,.
	\end{equation}
\end{corollary}

\paragraph{Generalization to a clique graph.}
Precisely as we approximated the (random) number of triangles $T_n$  by $\bar{T}_n \, = \,\frac{4 \left( \sum_{i\in\cE_n} X_i\right)^3}{3 n^3}$, we can approximate the random number of cliques with $\ell\geq 3$ vertices, by \footnote{For a fixed graph $G$ with $n$ vertices, this is equivalent to say that we approximate the number of cliques by $\bar{K}_n(G)=\frac{n^{\ell}(t(H_1,G))^{\binom{\ell}{2}}}{\aut}$. If $\ell=3$ we recover the definition of $\bar{T}_n(G)$.}
	\begin{equation}\label{Kn}
		\bar{K}_n:=\left(\frac{2E_n}{n^2}\right)^{\binom{\ell}{2}} \cdot \frac{n^\ell}{\aut}= \left(\frac{2 \sum_{i\in\cE_n} X_i}{n^2}\right)^{\binom{\ell}{2}} \cdot \frac{n^\ell}{\aut}.
	\end{equation}
 We then recover the following generalizations of Thms~\ref{Thm_LLN_mf}--\ref{Thm_non-standard_CLT}.
 
\begin{theorem}\label{Thm_LLN_mf_g}
Fix $\ell\in\mathbb{N}$, with $\ell\geq 3$. For all $(\alpha,h)\in\mathcal{U}^{rs}$, 
	\begin{equation}\label{SLLN_gen}
	\frac{\aut \bar{K}_n}{n^{\ell}} \, \xrightarrow{\;\;\mathrm{a.s.}\;\;}{} \, {u^{*}_0}^{\binom{\ell}{2}}(\alpha,h) \quad \text{ w.r.t. } \bar{\mathbb{P}}_{\alpha,h}, \text{ as } n \to +\infty,
	\end{equation}
where $u^*_0$ solves the maximization problem in \eqref{free_energy}.
For all $(\alpha,h) \in \mathcal{M}^{rs}$,
	\begin{equation}\label{conv_distr_gen}
	\frac{\aut \bar{K}_n}{n^{\ell}} \, \xrightarrow{\;\;\mathrm{d}\;\;}{} \, \kappa \delta_{u_{1,\ell}^{*}(\alpha,h)}+
	(1-\kappa)\delta_{u_{2,\ell}^{*}(\alpha,h)}  \quad \text{ w.r.t. } \bar{\mathbb{P}}_{n;\alpha,h}, \text{ as } n \to +\infty,
	\end{equation}
where $u_{1,\ell}^{*}:={u_1^{*}}^{\binom{\ell}{2}}(\alpha,h)$, $u_{2,\ell}^{*}:={u_2^{*}}^{\binom{\ell}{2}}(\alpha,h)$, $u_1^*$, $u_2^*$ solve the maximization problem in \eqref{free_energy}, and
$\kappa$ is given in \eqref{kappa}.
\end{theorem}

	\begin{theorem}[Generalized CLT w.r.t. $\bar{\mathbb{P}}_{n;\alpha,h}$]\label{generalized_Thm_standard_CLT}
Let $\bar{m}^{K}_n(\a,h):=\frac{\aut\bar{\E}_{n;\a,h}\left(\frac{\bar{K}_n}{n^{\ell-2}}\right)}{n^2}$, with $\ell\in\mathbb{N}$, $\ell\geq 3$. If $(\a,h) \in \mathcal{U}^{rs}\setminus \{(\a_c,h_c)\}$, then
		\[	
		\sqrt{\aut} \, \frac{\frac{\bar{K}_n}{n^{\ell-2}} - \frac{n^{2}}{\aut}\bar m_n^{K}(\a,h)}{n} \xrightarrow{\;\;\mathrm{d}\;\;}{} \mathcal{N}(0,\bar v^{K}(\alpha,h)) \quad \text{ w.r.t. } \bar{\mathbb{P}}_{n;\alpha,h},\text{ as } n \to +\infty,
		\]
		where $\mathcal{N}(0,\bar v^{K}(\alpha,h))$ is a centered Gaussian distribution with
		variance 	
		\[
		\bar v^{K}(\alpha,h) := \Big(\tbinom{\ell}{2}{u_0^*}^{\binom \ell2-1}\Big)^2 (2\aut c_0)^{-1} = \frac{\big(\tbinom{\ell}{2}{u^*}^{\binom{\ell}{2}-1}\big)^2}{2\aut}
		\frac{4u^*_0(1-u^*_0)}{1-2\alpha {u^*_0}^2(1-u_0^*)}.
		\]
 If $(\a,h)=\{(\a_c,h_c)\}$, then
$$
\aut \, \frac{\frac{\bar{K}_n}{n^{\ell-2}} - \frac{n^{2}}{\aut}\bar m^{K}_n(\a_c,h_c)}{n^{3/2}} \xrightarrow{\;\;\mathrm{d}\;\;}{} \bar W \quad \text{ w.r.t. } \bar{\mathbb{P}}_{n;\alpha_c,h_c},\text{ as } n \to +\infty,
$$
where $\bar W$ is a generalized Gaussian random variable with Lebesgue density $\bar{\ell}_{K}^c(y)\propto e^{-(\gamma y)^{4}}$, with $\gamma:= \frac{2^{\binom{\ell}{2} + 1/2}}{3^{\binom{\ell}{2}}}\binom{\ell}{2}$.	

\end{theorem}

\paragraph{Conditional measures.}\label{condm}
When $(\a,h)$ lies in the multiplicity curve $\cM^{rs}$,
where the solution of \eqref{FixPointEq} is not unique, we can still characterize the limiting behavior of the triangle density in the mean-field approximation, provided that we constraint the edge density to be close to one of the maximizers of the scalar problem \eqref{free_energy}.  To this aim, we consider a conditioned model, as follows. For $(\alpha,h) \in \mathcal{M}^{rs}$, let $u_i^*(\alpha,h)$ ($i=1,2$)  be the solutions of the scalar problem \eqref{free_energy}.
For $n\in\N$ and any fixed $\delta \in (0,1)$, consider the event
\begin{equation}\label{Bi_primo}
B_{u_i^{*}}\equiv B_{u_i^{*}}(n,\delta):=\left\{ x\in \mathcal{A}_n \,:\,\left|\frac{2E_n(x)}{n^2} - u_i^*(\alpha,h)\right|\le n^{-\d} \right\},
\end{equation} 
and define the conditional probability measures
\begin{equation}\label{conditional_measure}
\hat{\mathbb{P}}_{n;\alpha,h}^{(i)}\left( \,\cdot\, \right) := \bar{\mathbb{P}}_{n;\alpha,h} \left(\,\cdot\, \left\vert B_{u_i^{*}}(n,\delta)  \right.\right), \qquad \text{ for } i= 1,\,2\,.
\end{equation}
We denote the corresponding averages by $\hat{\mathbb{E}}_{n;\alpha,h}^{(i)}$ and we set  
$\hat{m}_n^{(i)}(\alpha,h):=\hat{\mathbb{E}}^{(i)}_{n;\alpha,h}\left(\frac{6\bar{T}_n}{n^3}\right)$.
The next statements represent the analog of the results presented in Subsec.~\ref{mean-field-results}, when the parameters belong to the multiplicity region $\cM^{rs}$. 

\begin{theorem}[Conditional SLLN and CLT]
	\label{Thm_conditional_limit_theorems_mfm}
	For $i=1,2$ and for all $(\alpha,h) \in \mathcal{M}^{rs}$, 
	\begin{equation}\label{conditional_lln_mfm}
		\frac{6\bar{T}_n}{n^3} \, \xrightarrow{\;\;\mathrm{a.s.}\;\;}{} \, u_i^{*}(\alpha,h)   \quad \text{ w.r.t. } \hat{\mathbb{P}}^{(i)}_{\alpha,h}, \text{ as } n \to +\infty,
	\end{equation}
	and
	\begin{equation}\label{conditional_clt_mfm}
		\sqrt{6} \, \frac{\frac{\bar{T}_n }{n}- \frac{n^2}{6} \hat{m}_n^{(i)}(\alpha,h)}{n} \, \xrightarrow{\;\;\mathrm{d}\;\;}{} \, 
		\mathcal{N}(0,\bar{v}_i^{\Delta}(\alpha,h)) 
		\quad \text{ w.r.t. } \hat{\mathbb{P}}^{(i)}_{n;\alpha,h}, \text{ as } n \to +\infty,
	\end{equation}
	where $\mathcal{N}(0,\bar{v}_i^{\Delta}(\alpha,h)) $ is a centered Gaussian distribution with variance $\bar{v}_i^{\Delta}(\alpha,h):= \frac{3{{u^4_i}^*(\a,h)}}{4c_i}$, being
	$c_i\equiv c_i(\alpha,h):=\frac{1-2\alpha[u_i^*(\alpha,h)]^2[1-u_i^*(\alpha,h)]}{4u_i^*(\alpha,h)[1-u_i^*(\alpha,h)]}.$
\end{theorem}

\begin{proposition}\label{prop_speed_conditional_mfm}
	For $i=1,2$ and for all $(\alpha,h) \in \mathcal{M}^{rs}$,
	\begin{equation}\label{speed_1_conditional_mfm}
		\lim_{n \to +\infty} n \cdot \hat{\mathbb{E}}_{n;\alpha,h}^{(i)} \left( \left\vert \frac{6\bar{T}_n}{n^3} - {u_i^*}^3(\alpha,h) \right\vert \right) = 
		\mathbb{E}\left(\left\vert \Bar{X}^{(i)} \right\vert\right),
	\end{equation}
	$\bar{X}^{(i)}$ is a centered Gaussian random variable with variance $6\bar{v}_i^{\Delta}(\alpha,h) = \frac{9{u_i^*(\a,h)}^4}{2c_i}$, being
	$c_i\equiv c_i(\alpha,h)=\frac{1-2\alpha[u_i^*(\alpha,h)]^2[1-u_i^*(\alpha,h)}{4u_i^*(\alpha,h)[1-u_i^*(\alpha,h)]}.$
\end{proposition}

\begin{corollary}\label{cor_speed_conditional_mfm}
	For $i=1,2$ and for all $(\alpha,h) \in \mathcal{M}^{rs}$,  we have
	\[
	\lim_{n \to +\infty} n \cdot \left(\hat{m}_n^{(i)}(\alpha,h)-{u_i^*}^3(\alpha,h)\right) = 0,
	\]
where we recall that $\hat{m}_n^{(i)}(\alpha,h):=\hat{\mathbb{E}}^{(i)}_{n;\alpha,h}\left(\frac{6\bar{T}_n}{n^3}\right)$.
\end{corollary}

\paragraph{Beyond the edge-triangle model.}
All the results stated in Subsec.~\ref{RET} can be easily extended to the Gibbs measure $\mathbb{P}_{\beta_k,\beta_1}$, whose associated Hamiltonian is obtained by setting $\beta_i=0$ for all $i\neq 1, k$ (with $k>2$) in \eqref{Hamiltonian}, i.e.
\begin{equation*}
\mathcal{H}_{n;\beta_k,\beta_1}(G):=n^2\left[\beta_{k}t(H_{k},G) +\beta_1 t(H_{1},G)\right]\,,
\quad \mbox{ for } G\in\cG_n.
\end{equation*}
The limiting free energy is given by the scalar maximization problem \eqref{scalar_probl} and the replica symmetric region is defined by condition $\beta_k\geq 0$ combined with \eqref{beta_piccoli_rs}.
The phase diagram has been fully characterized in \cite[Prop. 3.2]{RY} and is completely analogous to those represented in Fig.~\ref{fig:phase_diagram} for the edge-triangle model. More precisely,  problem \eqref{scalar_probl} has exactly one maximizer $u^*_0(\beta_k,\beta_1)$ on the whole replica symmetric regime, and exactly two maximizers $u^*_1(\beta_k,\beta_1)$,  $u^*_2(\beta_k,\beta_1)$ along a certain critical curve $\mathcal{M}^{rs}$ that starts at the critical point $(\beta_k^c,\beta_1^c)=\frac 12 \left(\frac{p^{p-1}}{(p-1)^p}, \log(p-1)-\frac {p}{(p-1)}\right)$, being  $p:= E(H_k)$.
At criticality the unique maximizer coincides with $u^*_c:= u^*(\beta^{c}_k,\beta^{c}_1)=\frac{p-1}p$. \\
In case $H_k$ is a clique with $\ell>3$ vertices, all the results of Subsecs.~\ref{mean-field-results}--\ref{condm} can be extended to the mean-field measure $\bar{\mathbb{P}}_{n;\beta_k,\beta_1}$ associated with $\bar{\mathcal{H}}_{n;\beta_k,\beta_1}(G):=\frac{\beta_k\aut}{n^{\ell-2}}\bar{K}_n(G)+2\beta_1 E_n(G)$, being $$\bar{K}_n(G)=\frac{n^{\ell}(t(H_1,G))^{\binom{\ell}{2}}}{\aut}$$
(see \eqref{Kn}). An equivalent version of Lem.~\ref{lemma_Z}, which is a crucial results for all the proofs, can be obtained by studying the \emph{energy function}
$ g_{\beta_k,\beta_1}(m):= \beta_k m^{p} + \beta_1 m  - {I(m)}/{2}$, for $m \in \G_n$,
(which generalizes \eqref{g}), and exploiting its Taylor  expansion 
\begin{align*}
	\begin{cases}
		g_{\beta_k,\beta_1}(m)=g_{\beta_k,\beta_1}(u_i^*)-c_i(m-u^*_i)^2 + o(m-u^*_i)^3 &
		\mbox{ if } (\beta_k,\beta_1) \neq (\beta_k^c,\beta_1^c) \vspace{0.2cm} \label{taylor_g_beta} \\
		g_{\beta_k,\beta_1}(m)=g_{\beta_k,\beta_1}(u_c^*)-\frac{p^5}{48(p-1)^2}(m-u_c^*)^4 + o(m-u_c^*)^5 &
		\mbox{ if } (\beta_k,\beta_1) = (\beta_k^c,\beta_1^c),
	\end{cases}
\end{align*}
where $i\in\{0,1,2\}$ and $c_i(\beta_k,\beta_1):=- g''_{\beta_k,\beta_1}(u^*_i)/2= \frac{1-2p(p-1)\beta_k{u^*_i}^{p-1}(1-u^*_i)}{4u^*_i(1-u^*_i)}.$
 
\section{Proofs: edge-triangle model} \label{proofs_et}
Before proving our main results, we show some basic properties of the derivatives w.r.t.~$\alpha$ of the free energy $f_{\alpha,h}$ given in \eqref{free_energy}. It is convenient to consider the finite size free energy $f_{n;\alpha,h}$ and the limiting free energy $f_{\alpha,h}$ as functions of $\alpha$. 

\subsection{Derivatives of the free energy}
For the sake of readability, we set  $f_{n,h}(\alpha):=f_{n;\alpha,h}$. We point out that for every $n$, $f_{n,h}$ is a convex function of $\alpha$; indeed the second derivative of $f_{n,h}$ is positive, because it is a variance (see \eqref{var_rel}).  Furthermore, we recall that, for $(\alpha,h)$ in the replica symmetric regime,  $\lim_{n\to \infty}f_{n,h}(\alpha)$ exists and it is given by $f_{h}(\alpha):=f_{\alpha,h}$, which is defined in \eqref{free_energy}.  We are then under the assumptions of the following lemma. 	
\begin{lemma}[\cite{E}, Lemma V.7.5.] \label{ellis}
	Let $(f_n)_{n \in \N}$ be a sequence of convex functions on an open interval $A$ of $\R$ such that $f(t) = \lim_{n \to +\infty}f_n(t)$ exists for every $t\in A$. Let $(t_n)_{n \in \N}$ be a sequence in $A$ which converges to a point $t_0 \in A$. If $f_n'(t_n)$ and $f'(t_0)$ exist, then $\lim_{n \to +\infty}f_n'(t_n)$ exists and equals $f'(t_0)$. 
\end{lemma} 
Now let $(\alpha,h)\in \mathcal{U}^{rs}$, and  $m^{\Delta}(\alpha,h) := 6 f'_{h}(\alpha) = {u_0^*}^3(\alpha,h)$. As a corollary of Lem.~\ref{ellis}  we obtain 
	\begin{equation} \label{m_lim}
	\lim_{n \to +\infty} m^{\Delta}_n(\alpha,h) = \lim_{n \to +\infty}6f'_{n,h}(\alpha) = 6 f'_{h}(\alpha)= m^{\Delta}(\alpha,h).
	\end{equation}
The following locally uniform convergence also holds.
\begin{proposition}\label{th:unif-conv}
	Let $h \in \mathbb{R}$ and $\, \mathcal{U}^{rs}_h:=\{\a \in (-2, +\ \infty): (\a,h)\in \mathcal{U}^{rs}\}$. If $\, \mathcal{U}^{rs}_h$ is non empty, then 
	$f'_{n,h} \longrightarrow f'_h$, as $n \to +\infty$,
	locally uniformly on $\, \mathcal{U}^{rs}_h$.
\end{proposition}

\begin{proof}
	We prove the statement by  contradiction, therefore we assume that there exists a compact set $A\subset \mathcal{U}^{rs}_h$ such that the convergence is not uniform on $A$. If this is the case, there exists $\varepsilon>0$ and a subsequence $(f_{n_k; \alpha, h})_{k \in \N}$ of $(f_{n; \alpha, h})_{n \in \N}$ such that for all $k \in \N$, 
	\begin{equation}\label{contradd}
		\max_{\alpha \in A}|f'_{n_k, h}(\a) - f'_{h}(\a)|>\varepsilon.
	\end{equation} 
	We recall that, since $(\alpha,h)\in \mathcal{U}^{rs}$, these derivatives are continuous.
	From \eqref{contradd} we then deduce that for each $k\in \N$ there exists $\alpha_k \in A$ such that 
	\begin{equation} \label{eps}
		|f'_{n_k,h}(\alpha_k)-f'_{h}(\alpha_k)|>\varepsilon.
	\end{equation}
	We can pass to a converging subsequence $(\alpha_{k_\ell})_{\ell \in \N}$ by compactness of $A$. If we denote by $\alpha_0$ the limit of this subsequence, by continuity  we have $f'_{n_{k_\ell},h}(\alpha_{k_\ell})  \longrightarrow f'_{h}(\alpha_0)$, as $\ell\to +\infty$. By triangle inequality, for all $\ell \in \N$,
	\begin{equation}
		\varepsilon<	|f'_{n_{k_\ell},h}(\alpha_{k_\ell})-f'_{h}(\alpha_{k_\ell})| \le |f'_{n_{k_\ell},h}(\alpha_{k_\ell})-f'_{h}(\alpha_0)|+|f'_{h}(\alpha_0)-f'_{h}(\alpha_{k_\ell})|.
	\end{equation} 	
	The l.h.s.~is bounded away from zero due to \eqref{eps}, while the r.h.s.~vanishes, as $\ell\to +\infty$, due to Lem.~\ref{ellis} and the pointwise convergence $f'_{n_{k_\ell},h}(\alpha_{k_\ell}) \longrightarrow f'_{h}(\alpha_0)$, as $\ell\to +\infty$. We get the contradiction.
\end{proof}	
In the next paragraph we provide a short overview of the main notions and results on graph limits theory, relevant to the proof of Thm.~\ref{Thm_convergence_in_distribution}. We refer the reader to \cite{BCLSV06,BCLSV08,BCLSV12,L12,LS06} for a thorough description of these concepts. 
\subsection{Key results on graph limits}\label{graphs_limit}
Let $(G_n)_{n \geq 1}$ be a sequence of simple graphs whose number of vertices tends to infinity;  the limit object of this sequence is a symmetric measurable function on the unitary square called \textit{graphon}. A crucial step for understanding where this definition comes from, is to introduce the notion of \textit{empirical graphon}.  
\begin{figure}[h!]
\centering
\includegraphics[scale=0.75]{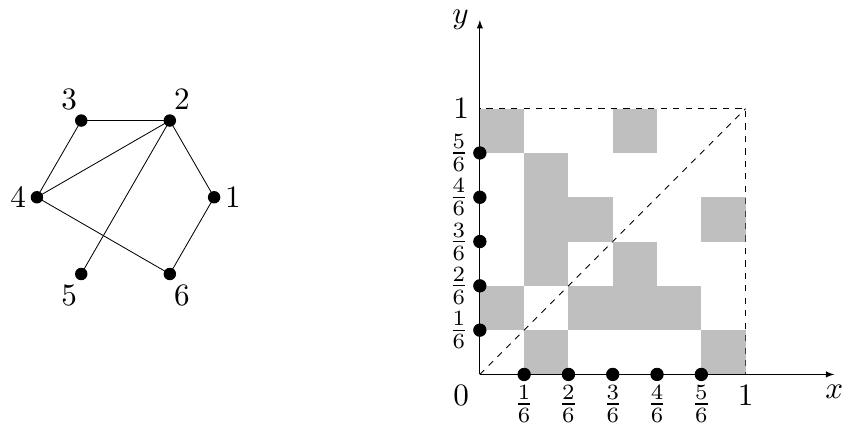}
\caption{ \small Graph $H$ with $m=6$ vertices on the left and its empirical graphon $g^{H}$ on the right. The gray regions are constantly equal to one, whereas the white regions are constantly equal to zero (example from \cite{H}). 
}
\label{graphon}
\end{figure}

Let $H$ be a finite simple graph $H$ with vertex set $[m]$. The empirical graphon $g^H$, corresponding to $H$, is defined by
\begin{equation}\label{graph_embedding}
g^H(x,y) := \left\{
\begin{array}{ll}
1 & \text{ if $\{\lceil mx \rceil, \lceil my \rceil\}$ is an edge in $H$}\\
0 & \text{ otherwise}
\end{array},
\right.
\end{equation}
where $(x,y) \in [0,1]^2$. In other words, $g^H$ is a step function corresponding to the adjacency matrix of $H$ (see Fig. \ref{graphon}). It is important to stress that any finite simple graph admits a graphon representation, therefore the sequence $(G_n)_{n \geq 1}$  can be equivalently turned into a sequence of empirical graphons $(g^{G_n})_{n \geq 1}$. Very intuitively, if we imagine to assign a black pixel to each block constantly equal to $1$ appearing in the step function $g^{H}$ (and conversely a white pixel to each block constantly equal to $0$), then, as $n$ gets large, pixels become finer and finer and the density of black pixels can be expressed as a number between $0$ and $1$. It is then more natural to see that the limit of $(g^{G_n})_{n \geq 1}$ can be represented by a measurable and symmetric function $g:[0,1]^2 \rightarrow [0,1]$ (called , indeed, graphon). The set of all graphons is denoted by $\mathcal{W}$; also notice that empirical graphons allow to represent all simple graphs as elements of the space $\mathcal{W}$.  The definition of convergence is formalized by making use of the notion of homomorphism density $t(H,G_n)$  \eqref{def_graph_hom_density} and its continuous analog, the so-called subgraph density
\begin{equation}\label{def_subgr_density}
t(H,g) := \int_{[0,1]^m} \prod_{\{i,j\} \in \mathcal{E}(H)} g(x_i,x_j) \, dx_1 \dots dx_m,
\end{equation}
where $\mathcal{E}(H)$ denotes the edge set of $H$. A sequence of graphs $(G_n)_{n \geq 1}$ is said to converge to the graphon $g$ if, for every finite simple graph $H$,
\begin{equation}\label{density_convergence}
\lim_{n \to +\infty} t(H,G_n) = t(H,g).
\end{equation}
Any sequence of graphs that converges in the appropriate way has a graphon as limit, and vice-versa every graphon arises as the limit of an appropriate graph sequence.
Intuitively, the interval $[0,1]$ represents a continuum of vertices and $g(x_i,x_j)$ corresponds to the probability of drawing the edge $\{x_i,x_j\}$. For instance, the \ER graphon is represented by the function that is constantly equal to $p$ on the unit square.
In order to take into account the arbitrary labeling of the vertices when they are embedded in the unit interval, one needs to introduce an equivalence relation on $\mathcal{W}$. Let $\Sigma$ be the space of all bijections $\sigma:[0,1] \rightarrow [0,1]$ preserving the Lebesgue measure. We say that the functions $g_1,g_2 \in \mathcal{W}$ are equivalent, and we write $g_1 \sim g_2$, if $g_2(x,y) = g_1(\sigma(x),\sigma(y))$ for some $\sigma \in \Sigma$. The quotient space under $\sim$ is denoted by $\widetilde{\mathcal{W}}$ and $\tau:g \mapsto \widetilde{g}$ is the natural mapping associating a graphon with its equivalence class. The space $\widetilde{\mathcal{W}}$ can be equipped with the so-called \textit{cut distance} that turns $\widetilde{\mathcal{W}}$ into a compact metric space (see \cite{LS07}, Thm.~5.1).  On the space $(\widetilde{\mathcal{W}},\delta_\square)$ a large deviation principle for the sequence of measures ($\widetilde{\mathbb{P}}_{n;p}^{\text{ER}})_{n \geq 1}$ of a dense \ER random graph has been proved by Chatterjee and Varadhan in \cite{CV11}. Here  $\widetilde{\mathbb{P}}_{n;p}^{\text{ER}}$ denotes the probability distribution on $\widetilde{\mathcal{W}}$ induced  by the \ER graph $G=G(n,p)$ via the mapping $G \mapsto g^G \mapsto \widetilde{g}^G$. We report below the large deviation principle:
\begin{theorem}[\cite{CV11}, Thm.~2.3]\label{Thm_LDP_ER}
For each fixed $p \in (0,1)$, the sequence $(\widetilde{\mathbb{P}}_{n;p}^{\text{ER}})_{n \geq 1}$ satisfies a large deviation principle on the space $(\widetilde{\mathcal{W}},\delta_{\square})$, with speed $n^2$ and rate function
\begin{equation}\label{ER_rate_function}
\mathcal{I}_p(\widetilde{g}) = \frac{1}{2} \int_0^1 \int_0^1  I_p(g(x,y)) \, dx \, dy,
\end{equation}
where $g$ is any representative element of the equivalence class $\widetilde{g}$ and, for $u \in [0,1]$, we set $I_p(u) = u \ln \frac{u}{p} + (1-u) \ln \frac{1-u}{1-p}$.
\end{theorem}
We will strongly rely on Thm.~\ref{Thm_LDP_ER} for the proof of Thm.~\ref{Thm_convergence_in_distribution}.

\subsection{Exponential convergence}
	
	\begin{proof}[Proof \hfill of \hfill Theorem~\ref{Thm_convergence_in_distribution}]
		The proof consists in showing that the sequence of the laws of the triangle density w.r.t.~the measure $\mathbb{P}_{n;\alpha,h}$ is exponentially tight; this is made by 
representing the measure $\mathbb{P}_{n;\alpha,h}$ as a tilted probability measure on the space of graphons, that has as a priori measure the \ER one. Let $H_2$ be a triangle. Note that the homomorphism density defined in \eqref{def_graph_hom_density} is then $t(H_2,G)=\frac{6T _n(G)}{n^3}$.
An important property that we are going to use is that $t(\,\cdot\,,G)= t(\,\cdot\,,\tilde{g})$, where $\widetilde{g}$ is the image in $\widetilde{\mathcal{W}}$ of the empirical graphon  $g^{G}$ of $G$ and $t(\cdot, \tilde{g})$ is the subgraph density \eqref{def_subgr_density}. This allows us to extend the edge-triangle Hamiltonian $\mathcal{H}_{n;\alpha,h}$ to the space $\widetilde{\mathcal{W}}$ replacing  homomorphism densities with subgraph densities. Indeed, for all $G\in\cG_n$, 
		\begin{equation}\label{PG}
		\P_{n;\alpha,h}(G)= 
		\frac{\exp \left(\mathcal{H}_{n;\alpha,h}(G)\right)}{
			\, \sum_{\widetilde{g} \in \widetilde{\mathcal{W}}} \, 
			\sum_{G \in [\, \widetilde{g} \,]_n} \exp \left(\mathcal{H}_{n;\alpha,h}(G)\right)}
		=
		\frac{\exp \left(\mathcal{H}_{n;\alpha,h}(\widetilde{g})\right)\1(g^G \in \widetilde{g})}{
			\, \sum_{\widetilde{g} \in \widetilde{\mathcal{W}}} \, 
			|[\, \widetilde{g} \,]_n|\exp \left(\mathcal{H}_{n;\alpha,h}(\widetilde{g})\right)}\,,
		\end{equation}
		where $[\, \widetilde{g} \,]_n:=\{G \in \mathcal{G}_n: g^G \in \widetilde{g}\}$ and $|\cdot|$ denotes the cardinality of a set. Notice that thanks to fact that we can replace the homomorphism density with the subgraph density, the internal sum $\sum_{G \in [\, \widetilde{g} \,]_n}$ in the second term of \eqref{PG} simply reduces to the cardinality of the set $[\, \widetilde{g} \,]_n$. Since for $p=\frac{1}{2}$ the \ER measure becomes uniform on $\mathcal{G}_n$, we can equivalently write $|[\, \widetilde{g} \,]_n| = 2^{n(n-1)/2}\mathbb{P}_{n;\frac{1}{2}}^{\text{ER}}([\, \widetilde{g} \,]_n)$. As a consequence, from \eqref{PG}, we obtain 
		\begin{equation}\label{repr_pf}
			\P_{n;\alpha,h}(G)=\frac{2^{-n(n-1)/2} \exp \left(\mathcal{H}_{n;\alpha,h}(\widetilde{g})\right)\1(g^G \in \widetilde{g})}
			{\sum_{\widetilde{g} \in \widetilde{\mathcal{W}}} \exp\left(\mathcal{H}_{n;\alpha,h}(\widetilde{g})\right) \widetilde{\mathbb{P}}_{n;\frac{1}{2}}^{\text{ER}}(\{\widetilde{g}\})}\,.
		\end{equation}
We now express the Hamiltonian in terms of homomorphism densities; to do so, we introduce the function
		\begin{equation}
			\label{U}
			U_{\alpha, h}(G) := \tfrac{\alpha}{6} t(H_2,G) + \tfrac{h}{2} t(H_1,G),
		\end{equation}
so that  $\mathcal{H}_{n;\alpha,h}(G) = n^2 U_{\alpha, h}(G)$. For each $n \geq 1$ and each Borel  set $\widetilde{A} \subseteq \widetilde{\mathcal{W}}$, we define the probabilities
		\begin{equation}\label{tilted_measures}
			\widetilde{\mathbb{Q}}_{n;\alpha,h}(\widetilde{A}) := \frac{\sum_{\widetilde{g} \in \widetilde{A}} \exp \left(n^2U_{\alpha, h}(\widetilde{g})\right) \widetilde{\mathbb{P}}_{n;\frac{1}{2}}^{\text{ER}}(\widetilde{g})}{\sum_{\widetilde{g} \in \widetilde{\mathcal{W}}} \exp\left(n^2U_{\alpha, h}(\widetilde{g})\right) \widetilde{\mathbb{P}}_{n;\frac{1}{2}}^{\text{ER}}(\widetilde{g})}.
		\end{equation}
		Since $U_{\alpha, h}$ is a continuous and bounded function on the metric space $(\widetilde{W},\delta_\square)$ (see~\cite{BCLSV08,BCLSV12}), by \cite[Thm.~II.7.2]{E},  the sequence $\{\widetilde{\mathbb{Q}}_{n;\alpha,h}\}_{n \geq 1}$ of probability measures satisfies a large deviation principle with speed~$n^2$ and rate function
		\begin{equation}\label{tilted_rate_function}
			\mathcal{I}_{\alpha,h}(\widetilde{g}) := \mathcal{I}_{\frac{1}{2}}(\widetilde{g}) - U_{\alpha, h}(\widetilde{g}) - \inf_{\widetilde{g} \in \widetilde{\mathcal{W}}} \left\{ \mathcal{I}_{\frac{1}{2}}(\widetilde{g}) - U_{\alpha, h}(\widetilde{g}) \right\}.
		\end{equation}
The function $\mathcal{I}_{\frac{1}{2}}$ is lower semicontinuous (see \cite[Lem.~2.1]{CV11}), and, therefore $\mathcal{I}_{\alpha,h}$ is lower-semincontinuous as well, as it is the sum of lower-semincontinuous functions; thus it admits a minimizer on the compact space $\widetilde{\mathcal{W}}$. In particular, for $(\alpha, h) \in \mathcal{M}^{rs}$  the minimizers of \eqref{tilted_rate_function} are given by the set $\widetilde{C}^*=\{\widetilde{u}^*_1, \widetilde{u}^*_2\}$, where $\widetilde{u}^*_1$ and $\widetilde{u}^*_2$ are the images in $\widetilde{\mathcal{W}}$ of the solutions ${u}^*_1, {u}^*_2$ to the scalar problem \eqref{free_energy} (we know that they are exactly two thanks to \cite[Prop.~3.2]{RY}). For all sufficiently small  $\varepsilon>0$, we define the open interval
	\begin{equation*} 
	J(\varepsilon):= ({u_1^{*}}^{3}-\varepsilon,{u_1^{*}}^3+\varepsilon) \cup  ({u_2^{*}}^3-\varepsilon,{u_2^{*}}^3+\varepsilon)
	\end{equation*}
	and we consider the sets
	\begin{equation*}
	\widetilde{C}^*_{\varepsilon} := \{\tilde{g}\in \widetilde{\mathcal{W}}: t(H_2,\widetilde{g})\notin J(\varepsilon) \}
	\quad  \mbox{and} \quad
	C^*_{\varepsilon} :=  \left\{G\in \mathcal{G}_{n}: \frac{6T _n(G)}{n^{3}}\notin J(\varepsilon)\right\}.
	\end{equation*}
It is important to observe that, due to \eqref{repr_pf} and \eqref{tilted_measures},
 $\widetilde{\mathbb{Q}}_{n;\alpha,h}(\widetilde{C}^*_{\varepsilon}) = \mathbb{P}_{n;\alpha,h}\left(C^*_{\varepsilon} \right)$. Moreover,
	 $\widetilde{C}^*_{\varepsilon}$ does not contain any element of $\widetilde{C}^*$, indeed, for the constant graphons $\widetilde{u}^*_i, i\in\{1,2\}$, it holds $t(H_2,\widetilde{u}^*_i)={{u}^*_i}^3 \in J(\varepsilon) \Rightarrow \widetilde{C}^*_{\varepsilon}\cap \widetilde{C}^*=\emptyset$. Hence, since $\widetilde{C}^*_{\varepsilon}$ is closed and does not contain any minimizer of \eqref{tilted_rate_function}, Thm.~II.7.2(b) in \cite{E} guarantees that, for sufficiently large $n$, there is some positive constant $\tau=\tau(\widetilde{C}^*_{\varepsilon})$ such that $\widetilde{\mathbb{Q}}_{n;\alpha,h}(\widetilde{C}^*_{\varepsilon}) \leq e^{-n^2 \tau}$.
		The thesis follows since 
		\begin{equation*}
			\mathbb{P}_{n;\alpha,h}\left(\frac{6T_n}{n^3}\in J(\varepsilon)\right)= 1- \widetilde{\mathbb{Q}}_{n;\alpha,h}(\widetilde{C}^*_{\varepsilon}) \ge 1-e^{-n^{2}\tau}\,. \qedhere
		\end{equation*}
	\end{proof}
When $(\alpha,h) \in \mathcal{U}^{rs}$, namely we work in the uniqueness regime, the proof above can be carried out exactly in the same way, and it gives exponential convergence of the sequence $\left( 6T_n/n^3\right)_{n \geq 1}$ to ${u_0^*}^{3}$. Indeed in $\mathcal{U}^{rs}$, the set of minimizers of \eqref{tilted_rate_function} coincides with the singleton $\widetilde{C}^*=\{\widetilde{u}_0^*\}$, where $\widetilde{u}_0^*$ is the image in $\widetilde{\mathcal{W}}$ of the unique solution $u_0^*$ to the scalar problem \eqref{free_energy}.
As pointed out in \cite[Thm.~\ref{Thm_convergence_in_distribution}]{BCM},  as a byproduct of this proof, we obtain an LDP for $\widetilde{\mathbb{Q}}_{n;\alpha,h}$:
\begin{remark}[\cite{BCM}, Rem.~7.5]\label{Rmk_tilted_LDP}
The sequence $(\widetilde{\mathbb{Q}}_{n;\alpha,h})_{n \geq 1}$ obeys a large deviation principle on the space $(\widetilde{W},\delta_{\square})$, with speed $n^2$ and rate function $\mathcal{I}_{\alpha,h}$. 
\end{remark}	
	\subsection{SLLN}
	We are now ready to prove the strong law of large numbers stated in Thm.~\ref{Thm_LLN}.
	\begin{proof}[Proof of Theorem \ref{Thm_LLN}]
The thesis immediately follows as a consequence of Borel-Cantelli  lemma, since exponential convergence provided by Thm.~\ref{Thm_convergence_in_distribution} implies almost sure convergence (see \cite{E}, Thm.~II.6.4). We stress that this almost sure convergence holds w.r.t.~a probability measure $\P_{\a,h}$ on the space $\left(\{0,1\}^{ \N},\mathcal{B}(\{0,1\}^{\N})\right)$,  with marginals corresponding to the measures $\P_{n;\a,h}$, for all $n\in\N$ (see Rem.~\ref{consist_cond}). 
	\end{proof}
\paragraph{Generalization to a generic simple graph.}
\begin{proof}[Proof of Thm.~\ref{genHk}]
The first part of the proof follows exactly the same steps as Thm.~\ref{Thm_convergence_in_distribution}.  Then, we define 
the open interval 
	\begin{equation*} 
	J(\varepsilon):= ({u_1^{*}}^{E(H_k)}-\varepsilon,{u_1^{*}}^{E(H_k)}+\varepsilon) \cup  ({u_2^{*}}^{E(H_k)}-\varepsilon,{u_2^{*}}^{E(H_k)}+\varepsilon)
	\end{equation*}
	and we consider the sets
	\begin{equation*}
	\widetilde{C}^*_{\varepsilon} := \{\tilde{g}\in \widetilde{\mathcal{W}}: t(H_k,\widetilde{g})\notin J(\varepsilon) \}
	\quad  \mbox{and} \quad
	C^*_{\varepsilon} :=  \left\{G\in \mathcal{G}_{n}: t(H_k, G) \notin J(\varepsilon)\right\}.
	\end{equation*}
Consider now the constant graphons $\widetilde{u}^*_i, i\in\{1,2\}$. As it happens for the triangle subgraph density, it holds $t(H_k,\widetilde{u}^*_i)={{u}^*_i}^{E(H_k)} \in J(\varepsilon) \Rightarrow \widetilde{C}^*_{\varepsilon}\cap \widetilde{C}^*=\emptyset$.
Notice that if $H_k$ is, say, a path of length $E(H_k)$, or a $E(H_k)$-star \footnote{
A $E(H_k)$-star is an undirected graph with $E(H_k)$ edges meeting at a \textquotedblleft root" vertex.}, or a clique with $E(H_k)$ edges, this distinction is not registered at the level of $t(H_k,\widetilde{u}^*_i)$. Indeed, being $\widetilde{u}^*_i$ a constant function, the subgraph density \eqref{def_subgr_density} returns in both cases the same value ${{u}^*_i}^{E(H_k)}$. The only difference is visible in the homomorphism density $t(H_k, G)$ that, depending on the subgraph we are taking into account, contains a different constant. Since $\widetilde{C}^*_{\varepsilon}$ is closed and does not include any minimizer of \eqref{tilted_rate_function}, by Thm.~II.7.2(b) in \cite{E} we finally conclude that 
$
\mathbb{P}_{n;\alpha,h}\left(t(H_k,\cdot)\in J(\varepsilon)\right)  \ge 1-e^{-n^{2}\tau},\,
$
being $\tau=\tau(\widetilde{C}^*_{\varepsilon})$ some positive constant.
\end{proof}

\begin{proof}[Proof of Thm.~\ref{Thm_LLN_Hk}]
When $(\alpha,h) \in \mathcal{U}^{rs}$, namely we work in the uniqueness regime, the proof of Thm.~\ref{genHk} gives exponential convergence of the sequence $\left( t(H_k, \cdot)\right)_{n \geq 1}$ to ${u_0^*}^{E(H_k)}$. The almost sure convergence immediately follows as a consequence of Borel-Cantelli  lemma, as for the proof of  Thm.~\ref{Thm_LLN}.
\end{proof}

\section{Proofs: mean-field approximation} \label{proofs-mf}
\subsection{Preliminaries}
We start from some preliminaries that are preparatory to the proofs of this section. First, we observe that the Hamiltonian $\bar{\mathcal{H}}_{n;\a,h}$ (given in \eqref{Hamilt_MF}), which is defined on  $\mathcal{A}_n$, is actually a function of the edges density $m\equiv m(x):=2E_n(x)/n^2$,  $x \in \cA_n$, taking values in the set $\G_n:=\left\{0, \frac{2}{n^2},\frac{4}{n^2},\frac{6}{n^2}, \dots, 1-\frac{1}{n}\right\}$. In particular, for all $x\in\cA_n$ such that $\frac{2E_n(x)}{n^2}= m$,
we have
\be\label{H_mf2}
\bar{\mathcal{H}}_{n;\alpha,h}(x)= \bar{\mathcal{H}}_{n;\alpha,h}(m)= n^2 \left(\frac{\alpha}{6} m^{3} + \frac{h}{2} m \right).
\ee
As a consequence, we can also write, with a little abuse of notation,
	\be \label{mf-hamiltonian-scalar}
	\bar{\mathbb{P}}_{n;\a,h}(A)=\sum_{m \in A}\, \frac{\cN_m e^{n^2 (\frac{\alpha}{6}m^{3} + \frac{h}{2}m)}}{\bar{Z}_{n;\alpha,h}} \,, \quad \text{ for } A \subseteq \G_n,
	\ee
	where $\cN_m:=\binom{\frac{n(n-1)}{2}}{\frac{n(n-1)m}{2}}$ coincides with the number of adjacency matrices in $\cA_n$ with edge density $m$.
From representation \eqref{mf-hamiltonian-scalar} arises another important function, that we call \emph{energy function}; for any $(\alpha,h) \in \mathbb{R}^2$, it is defined as
\be\label{g}
g_{\a,h}(m):= \frac{\alpha}{6}m^{3} + \frac{h}{2}m - \frac{I(m)}{2} \, , \quad \text{ for } m \in \G_n.
\ee
The first two terms coincide with the exponent  in \eqref{mf-hamiltonian-scalar}, whereas the entropic term $I(m)$, defined below \eqref{scalar_probl}, comes from the Stirling approximation of the binomial coefficient $\cN_m$. 
Let $f_{\a,h}$ be the infinite volume 
free energy of the edge-triangle model, as given in \eqref{free_energy}, and let $u^*_i$, $i=0,1,2$, or $u^*_c$ be a solution of \eqref{FixPointEq}, depending on the range of $(\alpha,h)$. We stress that, by construction,  $g_{\alpha,h}(u^*_i)=f_{\a,h}$ if $(\alpha,h)\neq (\alpha_c,h_c)$, and $g_{\alpha_c,h_c}(u^*_c)=f_{\a_c,h_c}$.
\paragraph{Neighborhoods of the maximizer(s).} Fix $0<\delta<1$. We will mainly work in the following neighborhoods of the maximizer(s) (whose definition was anticipated in \eqref{Bi_primo}):
	\begin{align}
		B_{u_i^*} &\equiv B_{u_i^*}(n,\d) = \{m\in\G_n\,:\, |m-u_i^*|\leq n^{-\d}\}, \quad i=0,1,2\, \label{Bi}\\
  		B_{u_c^*} &\equiv B_{u_c^*}(n,\d) = \{m\in\G_n\,:\, |m-u_c^*|\leq n^{-\d}\} \label{Bc}
	\end{align}
 making extensively use of the Taylor expansions 
\begin{align}
\begin{cases}
	g_{\a,h}(m)-g_{\a,h}(u_i^*)=-c_i(m-u^*_i)^2 + k_i(m-u^*_i)^3 &
	\mbox{ if } (\alpha,h) \neq(\a_c,h_c) \vspace{0.2cm} \label{taylor_g} \\
	g_{\a,h}(m)-g_{\a,h}(u_c^*)=-\frac{81}{64}(m-u_c^*)^4 + k_c (m-u_c^*)^5 &
	\mbox{ if } (\alpha,h) = (\a_c,h_c),
 \end{cases}
\end{align}
where 
\begin{align} \label{c}
		c_i &:= -\frac{g''_{\a,h} (u_i^*)}{2} = \frac{1-2\alpha(u_i^*)^2(1-u_i^*)}{4u_i^*(1-u_i^*)} >0, \\
		k_i&:= g'''_{\a,h} (\tilde{u}_i)/6, \quad  k_c:=g^{(v)}_{\a_c,h_c} (\tilde{u}_c)/5! \label{k},
\end{align}
for some $\tilde{u}_i,\tilde{u}_c$ such that $|\tilde{u}_i-u^*_i|< n^{-\d}$, $|\tilde{u}_c-u_c^*|< n^{-\d}$.
\paragraph{Lattice sets.}
As a result of suitable changes of variables, obtained as fluctuations of $m\in\Gamma_n$,
we will need to consider the following integration ranges:
\begin{align} 
		R^{(n)}_{i,\delta} &:= \left(-n^{1-\delta},n^{1-\delta}\right) \cap
		\left\{ -nu^*_i, -nu^*_i+\frac{2}{n},\dots, -nu^*_i+(n-1) \right\},  \label{Ri} \\
		R^{(n)}_{c,\delta} &:= \left(-n^{\frac 12-\delta},n^{\frac 12-\delta}\right) \cap
		\left\{ -\sqrt{n}u_c^*, -\sqrt{n}u_c^*+\frac{2}{n^{\frac 32}},\dots, -\sqrt{n}u_c^*+\frac {n-1}{\sqrt{n}} \right\},\label{Rc}
\end{align}
where $i=0,1,2$.
The following lemma shows how the main contribution to $\bar{Z}_{n;\alpha,h}$ is given by sums over these sets.
\begin{lemma}[\cite{BCM}, Lemma 8.1]\label{lemma_Z}
	Let $(\alpha,h) \in (-2,+\infty) \times \mathbb{R}$ and let $f_{\alpha,h}$ be the infinite volume free energy  of the edge-triangle model given in \eqref{free_energy}.
\begin{itemize}

\item[]		Fix  $\d\in (0,1)$. If $(\alpha,h) \neq(\a_c,h_c)$ and $c_i, k_i$ as in \eqref{c}--\eqref{k}, let
		\begin{align} \label{Di}
			D_i^{(n)} := \sum_{x\in R^{(n)}_{i,\delta}} \frac{2}{n} \frac{e^{- c_i x^2 + \frac{k_i}{n} x^3}}
			{\sqrt{\left(u_i^*+\frac{x}{n}\right) \left(1-u_i^* -\frac{x}{n}\right)}}, \quad \text{for $i=0,1,2$.}
		\end{align}
\item[]		Fix  $\d\in \left(0,\frac{3}{8} \right)$. If $(\alpha,h) = (\alpha_c, h_c)$ and $k_c$ as in \eqref{k}, let
		\begin{equation} \label{Dc}
			D_c^{(n)} := \sum_{x\in R^{(n)}_{c,\delta}} \frac{2}{n^{3/2}}\frac{e^{- \frac{81}{64} x^4 + \frac{k_c}{\sqrt{n}} x^5}}
			{\sqrt{\left(u_c^*+\frac{x}{\sqrt{n}}\right) \left(1-u_c^* -\frac{x}{\sqrt{n}}\right)}}.
		\end{equation}
\end{itemize}
	Then, as $n \to +\infty$,
	\be\label{claim}
	\bar{Z}_{n;\a,h}=  \frac{e^{n^2 f_{\a,h}}}{2\sqrt{\pi }} \left( D^{(n)}(\alpha,h)\right) (1+o(1)),
	\ee
	where
	\[
	D^{(n)}(\alpha,h) :=
	\begin{cases}
		D_0^{(n)} & \text{ if } (\alpha,h) \in \mathcal{U}^{rs}\setminus \{(\a_c,h_c)\} \\[.2cm]
		D_1^{(n)} + D_2^{(n)} & \text{ if }  (\alpha,h) \in \mathcal{M}^{rs} \\[.2cm]
		 D_c^{(n)} & \text{ if } (\alpha,h) = (\alpha_c,h_c)
	\end{cases}.
	\]
\end{lemma}
\begin{remark}
	Lemma~\ref{lemma_Z} directly proves \eqref{equiv_fren}.
\end{remark}
\begin{remark}\label{norm_w}
	The quantities defined in \eqref{Di}--\eqref{Dc} are Riemann sums with volume elements respectively given by $2/n$ and $2/n^{3/2}$. Indeed the points $x \in R^{(n)}_i$, $i={0,1,2}$ (resp. inside $R^{(n)}_c$) are evenly spaced with gaps of length $2/n$ (resp.~$2/{n}^{3/2}$). Hence,
\begin{align}
D_i^{(n)} &\xrightarrow{\; n \to +\infty \;} 
		D_i:=  2\sqrt{\pi\left[{1-2\alpha \left( u_i^*\right)^2(1-u_i^*)}\right]^{-1}}, \qquad i\in\{0,1,2\}  \label{Dilim}\\
 D_c^{(n)} &\xrightarrow{\; n \to +\infty \;} 
D_c:=  \frac{3}{\sqrt{2}}\int_{-\infty}^{\infty}e^{-\frac{81}{64}x^4} dx \approx  3.63\,, \label{Dclim}
\end{align}
where for \eqref{Dclim} we used $u^*_c=\frac{2}{3}$.
Later, these terms will play the role of normalization weights.
 \end{remark}
\subsection{SLLN}
\begin{proof}[Proof of Theorem~\ref{Thm_LLN_mf}.]
By retracing the same passages of Thm.~\ref{Thm_convergence_in_distribution} we can prove exponential convergence of the sequence $\left(6\bar{T}_n/n^3\right)_{n\in\mathbb{N}}$ to ${u_0^*}^3$. By multiplying and dividing \eqref{mf-hamiltonian-scalar} by $2^{\binom{n}{2}}$ we obtain the analog of representation \eqref{tilted_measures} for the measure $\bar{\mathbb{P}}_{n;\a,h}$:
	\be
	\bar{\mathbb{P}}_{n;\a,h}(A)= \frac{\sum_{m \in A}\, e^{n^2 (\frac{\alpha}{6}m^{3} + \frac{h}{2}m)}\mathbb{P}^{\text{ER}}_{n;\frac{1}{2}}(m)}{\sum_{m \in \Gamma_n}\, e^{n^2 (\frac{\alpha}{6}m^{3} + \frac{h}{2}m)}\mathbb{P}^{\text{ER}}_{n;\frac{1}{2}}(m)} \,, \quad \text{ for } A \subseteq \G_n.
	\ee
Indeed, $\mathbb{P}^{\text{ER}}_{n;\frac{1}{2}}(m)= \cN_m/2^{\binom{n}{2}}$, where we recall that $\cN_m=\binom{\frac{n(n-1)}{2}}{\frac{n(n-1)m}{2}}$ coincides with the number of adjacency matrices in $\cA_n$ with edge density $m$. The same steps of Thm.~\ref{Thm_convergence_in_distribution}, which in particular rely on Ellis result \cite[Thm.~II.7.2]{E}, lead then to 
\begin{equation}
			\bar{\mathbb{P}}_{n;\alpha,h}\left(\frac{6\bar{T}_n}{n^3}\in J(\varepsilon)\right)  \ge 1-e^{-n^{2}\tau}\,.
\end{equation}
Here, $J(\varepsilon):= ({u_0^{*}}^{3}-\varepsilon,{u_0^{*}}^3+\varepsilon)$, $\varepsilon>0$, being $u^{*}_0$ the unique maximizer of \eqref{free_energy} (recall that $(\alpha,h)\in \mathcal{U}^{rs}$), and $\tau$ a positive constant depending on the complement set $J^{c}(\varepsilon)$.
Finally, almost sure convergence follows from exponential convergence as a consequence of Borel-Cantelli  lemma (\cite{E}, Thm.~II.6.4). 
\end{proof}

\subsection{Phase coexistence on the critical curve}
\begin{proof}[Proof of Theorem~\ref{Thm_convergence_in_distribution_mf}.]
	We will determine the limit of 
 \begin{equation*}
 \mathbb{\bar{E}}_{n;\alpha,h} \left[ \varphi \left( 6\bar{T}_n/n^3 \right)\right],
 \end{equation*}
 for any continuous and bounded real function $\varphi$. First we observe that, since $m\equiv m(x)=\frac{2 E_n(x)}{n^2}$, from definition \eqref{def:bT_n} we obtain $\bar{T}_n(x)= \frac{n^3 m^3}{6}$. Then, using \eqref{mf-hamiltonian-scalar}, we get:
\begin{equation}\label{media}
	\bar{\E}_{n;\a,h}\left[\varphi \left(\frac{6\bar{T}_n}{n^3}\right) \right]
	= \sum_{m\in\G_n} \varphi(m^3) \, \frac{\cN_m e^{n^2 (\frac{\alpha}{6}m^{3} + \frac{h}{2}m)}}{\bar{Z}_{n;\alpha,h}}.
\end{equation}
	We split the sum in \eqref{media} over the sets $B_{u_1^*}$, $B_{u_2^*}$ given in \eqref{Bi}, and $\cC \equiv \mathcal{C}(n,\d) := \G_n \setminus \left(B_{u_1^*}\cup B_{u_2^*}\right)$,

considering the three contributions separately. 
First we observe that, from \cite[Chap.~2, Eq.~(2.11)]{FV}, there exists two positive constants $c$ and $C$ such that
	\begin{equation} \label{Stirling-rough}
		cne^{-\frac{n^2}2 I(m)}\le\cN_m \le Ce^{-\frac{n^2}2 I(m)}.
	\end{equation}
Whenever we work inside the sets $B_{u_i^*}$, $i={1,2}$, the bounds in \eqref{Stirling-rough} can be made more precise, because 
$n^{-2}\ll m \ll 1-n^{-2}$ and, consequently, $n^2m \to \infty$ and $n^2(1-m)\to \infty$, as $n\to\infty$. Hence, the following Stirling approximation is valid
\begin{equation}\label{Stirling}
		\mathcal{N}_{m}= \frac{e^{-\frac{n^2}{2}I(m)}}{n \sqrt{\pi m(1-m)}} (1+o(1))\,.
	\end{equation}
	This, together with Lemma~\ref{lemma_Z}, yields the following representation:
	\begin{equation*}
	\begin{split}
		\sum_{m\in B_{u_i^*}} \varphi(m^3) \frac{\cN_m e^{n^2 (\frac{\alpha}{6}m^{3} + \frac{h}{2}m)}}{\bar{Z}_{n;\a,h}}
		&= \sum_{m\in B_{u_i^*}} \frac 2 n \frac{\varphi(m^3)}{\sqrt{m(1-m)}}
		\frac{e^{-n^2 \left( f_{\a,h} - g_{\a,h}(m)\right)}}{D_1^{(n)}+D_2^{(n)}} (1+o(1))\,,
	\end{split}
	\end{equation*}
 where $i\in\{1,2\}$ and $g_{\alpha,h}$ is the energy function defined in \eqref{g}.
By performing the change of variable $x=n(m-u_i^*)$, and using the Taylor expansion \eqref{taylor_g}, we obtain:
	\be\nonumber
	\begin{split}
		\sum_{m\in B_{u_i^*}} \varphi(m^3) \frac{\cN_m e^{n^2 (\frac{\alpha}{6}m^{3} + \frac{h}{2}m)}}{\bar{Z}_{n;\a,h}}=&
		\sum_{x\in R_{i,\delta}^{(n)}} \frac 2 n \frac{\varphi((u_i^* +\frac{x}{n})^3)}{\sqrt{(u_i^* +\frac{x}{n})(1-u_i^* -\frac{x}{n})}}
		\frac{e^{-c_i x^2 + \frac{k_i}{n} x^3}}{D_1^{(n)}+D_2^{(n)}} (1+o(1))\\
		&\xrightarrow{\;\; n \to +\infty \;\;} \,
		\varphi({u_i^*}^3) \frac{D_i}{D_1+ D_2}\,,
	\end{split}
	\ee
	where $D_i$ is defined in \eqref{Dilim} and $R^{(n)}_{i,\delta}$, given in \eqref{Ri}, represents the range of values of $x$.
To conclude the analysis, we show that the sum over the remaining set $\cC $ provides no contribution in the limit. Outside the sets $B_{u_i^*}$, $i=1,2$, the Stirling approximation \eqref{Stirling} is not valid anymore. However, from \eqref{Stirling-rough} we deduce: 
	\begin{equation}\label{compl1}
		\sum_{m \in \cC} \cN_me^{n^2(\frac \a 6 m^3 +\frac h 2 m)}\le Cn\sum_{m \in \cC} e^{n^2g_{\a,h}(m)}< \frac C2 n^3 e^{n^2 \max_{m\in \cC}g_{\a,h}(m)}, 
	\end{equation}
where the last inequality is due to the fact that $\cC$ contains at most $\binom{n}{2}$ points.
Since $f_{\a,h}=g_{\a,h}(u^*_i)$, $i=1,2$ we obtain 
	\begin{equation}\label{compl2}
		e^{n^2 \max_{m\in \cC}g_{\a,h}(m)}=e^{n^2f_{\a,h}}e^{-n^2(g_{\a,h}(u^*_i)-\max_{m\in \cC}g_{\a,h}(m))}\le e^{n^2f_{\a,h}}e^{-kn^{2-2\d}},
	\end{equation}
	where $k>0$ is a constant that does not depend on $n$ and $\delta\in (0,1)$. The last inequality follows from the Taylor expansion \eqref{taylor_g}, exploiting the fact that $|m-u^*_i|>n^{-\d}$ for $i=1,2$ and for all $m\in\mathcal{C}$.
As a consequence, from  the rough Stirling approximation \eqref{Stirling-rough}
and the consequent  rough bound $\bar{Z}_{n;\a,h} \geq cn^{-1} e^{n^2 f_{\a,h}}$,
we get
	\begin{equation} \label{avC}
		\frac{\sum_{m \in \cC}\cN_me^{n^2(\frac \a 6 m^3 +\frac h 2 m)}}{\bar{Z}_{n;\a,h}} < c^{-1}Cn^4e^{-kn^{2-2\delta}} \xrightarrow{\;\; n \to +\infty \;\;} 0,
	\end{equation}
being $2-2\delta>0$ by assumption. In conclusion,
	$$\lim_{n\to+\infty} \bar{\E}_{n;\a,h}\left[\varphi \left(\frac{6\bar{T}_n}{n^3}\right) \right]
	= \varphi({u_1^*}^3) \frac{D_1}{D_1+ D_2}+\varphi({u_2^*}^3) \frac{D_2}{D_1+ D_2}\,.
	$$
	This proves the thesis.
\end{proof}
\subsection{Mean-field convergence error}
\begin{proof}[Proof of Proposition \ref{prop_speed_mmf}]
	Let $(\alpha,h) \in \mathcal{U}^{rs}$ and $u_0^*=u^*_0(\a,h)$ be the unique maximizer which solves the fixed point equation \eqref{FixPointEq}. The proof implements the same machinery of Thm.~\ref{Thm_convergence_in_distribution_mf}.
First we observe that 
\begin{equation}\label{avG}
\bar{\E}_{n;\a,h}\Big( \Big\vert \frac{6\bar{T}_n}{n^3} - {u_0^*}^3 \Big\vert\Big)
		= \sum_{m\in \Gamma_n}|m^3-{u_0^*}^3|
		\frac{\cN_m e^{n^{2}(\frac{\alpha}{6}m^{3} + \frac{h}{2}m)}}
		{\bar{Z}_{n;\a,h}}.
\end{equation}
We split the average in \eqref{avG} in two parts, one over $B_{u_0^*}$ given in \eqref{Bi}, and the other over $\cC \equiv \mathcal{C}(n,\d) := \G_n \setminus B_{u_0^*}$.
The contribution of the average over $\cC$ is negligible, exploiting exactly the same argument in \eqref{compl1}--\eqref{avC} with $u_0^{*}$ in place of $u_i^{*}$ ($i=1,2$), and bounding, very roughly,  $|m^3-{u_0^{*}}^3|=|m-u_0^{*}|(m^2 +mu_0^{*}+{u_0^{*}}^2)$ by the constant $3$.
We now focus on the sum over $B_{u_0^*}$. By Lem.~\ref{lemma_Z} and the Stirling approximation \eqref{Stirling}, we obtain
	\be\label{prima}
	\begin{split}
		\bar{\E}_{n;\a,h}\Big( \Big\vert \frac{6\bar{T}_n}{n^3} - {u_0^*}^3 \Big\vert\Big)
		&= \sum_{m\in B_{u_0^*}}|m^3-{u_0^*}^3|
		\frac{\cN_m e^{n^{2}(\frac{\alpha}{6}m^{3} + \frac{h}{2}m)}}
		{\bar{Z}_{n;\a,h}}(1+o(1))\\
		&= \sum_{m\in B_{u_0^*}} \frac 2 n\frac{|m^3-{u_0^*}^3|}{\sqrt{m(1-m)}}\frac{e^{-n^2 (f_{\a,h}-g_{\a,h}(m))}}
		{D^{(n)}_0}(1+o(1))\,,
	\end{split}
	\ee
 where we recall that $\bar{T}_n(x)= \frac{n^3 m^3}{6}$ and $ B_{u_0^*}$ is defined in \eqref{Bi}.
 First, we analyze the case $(\alpha,h) \in \mathcal{U}^{rs}\setminus \{(\a_c,h_c)\}$ and then we move to the critical point.
From the Taylor expansion~\eqref{taylor_g} we get
	\begin{equation}
		\label{seconda}
		\begin{split}
			\bar{\E}_{n;\a,h}\Big( \Big\vert \frac{6\bar{T}_n}{n^3} - {u_0^*}^3 \Big\vert\Big)&=
			\sum_{m\in B_{u_{0}^*}} \frac 2 n\frac{|m^3-{u_0^*}^3|}{\sqrt{m(1-m)}}
			\frac{e^{-n^{2}c_0(m-u_0^*)^2 + n^2 k_0(m-u_0^*)^3}}
			{D_0^{(n)}}(1+o(1)).
		\end{split}
	\end{equation}
We now perform the change of variable $x= n(m-{u_0^*})$, and we use the identity
	\begin{equation} \label{cubic-expansion}
		m^3-{u_0^*}^3=\left(u_0^*+\frac x {{n}}\right)^3-{u_0^*}^3= 3{u_0^*}^2 \frac{x}{{n}} +3 {u_0^*}\frac{x^2}{n^2}+\frac{x^3}{n^{3}},
	\end{equation}
thus obtaining
	\begin{equation*}
		\begin{split}
			n \cdot \bar{\E}_{n;\a,h}
			\Big( \Big\vert \frac{6\bar{T}_n}{n^3} - {u_0^*}^3 \Big\vert\Big)&=
			\sum_{x\in R_{0,\delta}^{(n)}} \frac 2 n
			\frac{| 3{u_0^*}^2 {x} +3 {u_0^*}\frac{x^2}{n}+\frac{x^3}{n^{2}}|
				\cdot e^{-c_0x^2 + \frac{k_0}{n}x^3 }}
			{\sqrt{\left(u_0^*+\frac{x}{{n}}\right) \left(1-u_0^* -\frac{x}{{n}}\right)}\cdot D_0^{(n)}}(1+o(1)),
		\end{split}
	\end{equation*}
	where $R_{0,\delta}^{(n)}$ is defined in \eqref{Ri} and the constants $c_0$ and $k_0$ are given in \eqref{c}--\eqref{k}. 
We can upper and lower bound the sum above using the following chain of inequalities: $|a|-|b| \leq ||a|-|b|| \leq |a+b| \leq |a|+|b|$, $a,b\in\mathbb{R}$, with $a:=3{u_0^*}^2 {x}$ and $b:=3 {u_0^*}\frac{x^2}{n}+\frac{x^3}{n^{2}}$. Consider the term   
\begin{equation}\label{a}
	\sum_{x\in R_{0,\delta}^{(n)}} \frac 2 n
			\frac{| 3{u_0^*}^2 {x}|
				\cdot e^{-c_0x^2 + \frac{k_0}{n}x^3 }}
			{\sqrt{\left(u_0^*+\frac{x}{{n}}\right) \left(1-u_0^* -\frac{x}{{n}}\right)}\cdot D_0^{(n)}}(1+o(1)),
\end{equation}
together with the sequence of probability densities
	\begin{equation}\label{dens_nc}
	\ell_n(x):= \frac 2 n
	\frac{
		e^{-c_0x^2 + \frac{k_0}{n}x^3 }}
	{\sqrt{\left(u_0^*+\frac{x}{{n}}\right) \left(1-u_0^* -\frac{x}{{n}}\right)}\cdot D_0^{(n)}}\1_{R_0^{(n)}} (x)\,,\quad x\in\R\,,
	\end{equation}
	where $D_0^{(n)}$ is the normalization weight defined in \eqref{Di}.
	If, for every $n \in \N$, $X_n$ is a random variable with density $\ell_n$, then
 \begin{equation}
	\sum_{x\in R_{0,\delta}^{(n)}} \frac 2 n
			\frac{| 3{u_0^*}^2 {x}|
				\cdot e^{-c_0x^2 + \frac{k_0}{n}x^3 }}
			{\sqrt{\left(u_0^*+\frac{x}{{n}}\right) \left(1-u_0^* -\frac{x}{{n}}\right)}\cdot D_0^{(n)}}(1+o(1))=\E(|3{u_0^*}^2X_n|)(1+o(1)).
\end{equation}
Collecting all contributions we have:
\begin{align}
&\E(|3{u_0^*}^2X_n|)(1+o(1)) - \left(\int_{\mathbb{R}}\left|3 {u_0^*}\frac{x^2}{n}+\frac{x^3}{n^{2}}\right|\ell_{n}(x) dx\right) (1+o(1)) \label{final_sum_a} \\
&\hspace{3cm  } \leq n \cdot \bar{\E}_{n;\a,h}
			\Big( \Big\vert \frac{6\bar{T}_n}{n^3} - {u_0^*}^3 \Big\vert\Big) \leq  \notag \\ 
&  \E(|3{u_0^*}^2X_n|)(1+o(1)) + \left(\int_{\mathbb{R}}\left|3 {u_0^*}\frac{x^2}{n}+\frac{x^3}{n^{2}}\right|\ell_{n}(x) dx\right) (1+o(1)).    \label{final_sum_b}
\end{align}
Notice that, due to the convergence $D_{0}^{(n)} \xrightarrow[]{n \to +\infty} D_{0}$ (see Rem.~\ref{norm_w}) and the Scheff\'e Lemma, we obtain \mbox{$X_n \xrightarrow[]{\;\; d \;\;} X$}, where $X$ 
is a Gaussian random variable with density
\begin{equation}\label{liml}
\ell(x)=  \sqrt{\frac{c_0}{\pi}}e^{-c_0x^2}\,,\quad x\in\R\,.
\end{equation}
Moreover, the random variables $X_n$ have finite exponential  moments for any
sufficiently large $n$.	Therefore, by the dominated convergence theorem, applied to both bounds in \eqref{final_sum_a}--\eqref{final_sum_b}, we obtain
$$n\cdot\bar{\E}_{n;\a,h}\Big( \Big\vert \frac{6\bar{T}_n}{n^3} - {u_0^*}^3 \Big\vert\Big)\xrightarrow{\;\; n \to +\infty \;\;} \, 3{u_0^*}^2\E(|X|).$$
Indeed, the second summand in terms \eqref{final_sum_a}--\eqref{final_sum_b} vanishes, being $3 {u_0^*}\frac{x^2}{n}+\frac{x^3}{n^{2}}=o(1)$, for fixed $x$.
Setting $\bar{X}:= 3{u_0^*}^2 X$, and noticing that $X$ has variance $(2c_0)^{-1}$, we recover \eqref{speed_mmf1}.
	\vspace{0.3cm} \\
	\noindent
We now move to the critical case, so we consider $(\a,h)=(\a_c,h_c)$ and $u_c^*=u^*(\a_c,h_c)$. Here, the proof works exactly the same.
We split the average in \eqref{avG} in two parts, one over $B_{u_c^*}$ given in \eqref{Bc}, and the other over $\cC \equiv \mathcal{C}(n,\d) := \G_n \setminus B_{u_c^*}$.
The contribution of the average over $\cC$ is negligible, exploiting the same argument in \eqref{compl1}--\eqref{avC}. This time, by injecting the Taylor expansion \eqref{taylor_g} at the critical point in \eqref{compl2}, and using the fact that \mbox{$|m-u_c^*|^4>n^{-4\d}$} for $m\in\cC$, we get 
	\begin{equation*} 
		\frac{\sum_{m \in \cC}\cN_me^{n^2(\frac{\alpha_c}{6} m^3 +\frac{h_c}{2} m)}}{\bar{Z}_{n;\a_c,h_c}} < c^{-1}Cn^4e^{-kn^{2-4\delta}} \xrightarrow{\;\; n \to +\infty \;\;} 0, \qquad k>0,
	\end{equation*}
since $\delta<3/8$ by assumption. 
We now restrict the average \eqref{avG} to a sum in $B_{u_c^*}$.
In place of \eqref{seconda} we get:	
	\begin{equation}\nonumber
		\begin{split}
			\bar{\E}_{n;\alpha_c,h_c}
			\Big( \Big\vert \frac{6\bar{T}_n}{n^3} - {u_c^*}^3 \Big\vert\Big)
			&=
			\sum_{m\in B_{u_c^*}}
			\frac 2 {n^{\frac 32}}
			\frac{|m^3-{u_c^*}^3|
				e^{-n^2\frac{81}{64}(m-u_c^*)^4 + n^2{k_c}(m-u_c^*)^5}}
			{\sqrt{m \left(1-m\right)}\cdot  D_c^{(n)}} (1+o(1)),		
		\end{split}
	\end{equation}
where $B_{u_c^*}$ is defined in \eqref{Bc}, and the constant $k_c$ is given in \eqref{k}.
Notice that here the Taylor expansion \eqref{taylor_g} provides the fourth-order term $(m-u_c^*)^4$ at the exponent, while Lem.~\ref{lemma_Z} brings the normalization weight $D_c^{(n)}$. After the change of variable \mbox{$y:=\sqrt{n}(m-{u_c^*})$}, recalling that $u_c^*=\frac 2 3$ and by means of the identity 
	\begin{equation}\label{cubic-expansion_c}
		m^3-{u_c^*}^3=\Big(u_c^*+\frac y {\sqrt{n}}\Big)^3-{u_c^*}^3= \frac 4 3 \frac{y}{\sqrt{n}} +2\frac{y^2}{n}+\frac{y^3}{n^{\frac32}},
	\end{equation}
we obtain
	\begin{equation}\nonumber
		\begin{split}
			\sqrt{n}\cdot\bar{\E}_{n;\alpha_c,h_c}
			\Big( \Big\vert \frac{6\bar{T}_n}{n^3} - {u_c^*}^3 \Big\vert\Big)
			&=
			\sum_{y\in R_{c,\delta}^{(n)}}
			\frac 2 {n^{\frac 32}}
			\frac{|\frac 4 3 {y} +2\frac{y^2}{\sqrt{n}}+\frac{y^3}{n}|
				\cdot	e^{-\frac{81}{64}y^4 + \frac{k_c}{\sqrt n} y^5}}
			{\sqrt{\big(u_c^*+\frac{y}{\sqrt{n}}\big) \big(1-u_c^* -\frac{y}{\sqrt{n}}\big)}  D_c^{(n)}} (1+o(1)),
		\end{split}
	\end{equation}
	where $R_{c,\delta}^{(n)}$ is given in \eqref{Rc}. For every $n\in\mathbb{N}$, let $Y_n$ be a real random variable with Lebesgue density
\be\label{densc}
\ell^c_n(y):= \frac{2}{n^{3/2}} 
\frac{e^{-\frac{81}{64}y^4  + \frac{k_c}{\sqrt{n}}y^5 }}
{\sqrt{(u_c^*+\frac{y}{\sqrt n})(1-u_c^*-\frac{y}{\sqrt n})} \cdot D_c^{(n)}} \, \1_{R_{c,\delta}^{(n)}} (y)\,,\quad y\in\R\, .
\ee
Notice that $D_c^{(n)}$ provides the right normalization rate. The random variable $Y_n$ has finite  exponential moments for any sufficiently large $n$.
 By Scheff\'e Lemma and dominated convergence theorem, we conclude 
 $$\sqrt{n}\cdot\bar{\E}_{n;\alpha_c,h_c}
			\Big( \Big\vert \frac{6\bar{T}_n}{n^3} - {u_c^*}^3 \Big\vert\Big)\xrightarrow{\;\; n \to +\infty \;\;} \, \frac{4}{3}\E(|Y|),
$$
where $Y$ is a generalized Gaussian random variable with Lebesgue density 
$\ell^c(y)\propto e^{-\frac{81}{64}y^4}$. Setting
 $\bar Y:=\frac 4 3 Y$, since the scale parameter of $Y$ is $\frac{2^{3/2}}{3}$, we deduce that $\bar{Y}$ is a generalized Gaussian random variable with scale parameter $\frac{2^{7/2}}{3^2}$, thus proving the thesis.
\end{proof}
With the same strategy we can immediately prove the following corollary.
\begin{proof}[Proof of Corollary \ref{cor_speed_mf}]
Recall that $\bar m^{\Delta}_n(\a,h)= \bar{\E}_{n;\a,h}\big( \frac{6\bar{T}_n}{n^3}\big)$. By following the proof of Prop.~\ref{prop_speed_mmf}, we obtain,
\begin{itemize}
\item for all
	$(\alpha,h) \in \mathcal{U}^{rs}\setminus \{(\a_c,h_c)\}$,
	$$n\cdot\left( \bar m^{\Delta}_n(\a,h)-{u_0^*}^3 (\a,h)\right)
	= \E\left(3 {u_0^*}^2 X_n\right)(1+o(1))  +o(1) \, \xrightarrow{\;\; n \to +\infty \;\;} \, \E(\bar X)=0,$$
\item for $(\alpha,h)=(\a_c,h_c)$ 
	$$\sqrt{n}\cdot\left( \bar m^{\Delta}_n(\a_c,h_c)-{u^*}^3(\a_c,h_c)\right)
	= \E\Big( \frac4 3Y_n\Big)(1+o(1)) +o(1) \, \xrightarrow{\;\; n \to +\infty \;\;} \, \E(\bar Y)=0.$$
 \end{itemize}
\end{proof}

\subsection{CLT (off the critical curve)}
We use  Cor.~\ref{cor_speed_mf} to prove Thms.~\ref{Thm_standard_CLT} and \ref{Thm_non-standard_CLT}.
We start with Thm.~\ref{Thm_non-standard_CLT} at the critical point.
\proof[Proof of Theorem~\ref{Thm_non-standard_CLT}]
Let $u_c^{*}=u^{*}(\alpha_c,h_{c})$. We consider the decomposition
\be\label{equivalenza}
6 \, \frac{\frac{\bar{T}_n}{n} - \frac{n^{2}}{6}\bar m^{\Delta}_n(\a_c,h_c)}{n^{3/2}}
= \bar U_n
+ \sqrt{n}\left({u_c^{*}}^3 - \bar m^{\Delta}_n(\a_c,h_c)\right),
\ee
where 
\begin{equation*}
	\bar{U}_n:= 6 \, \frac{\frac{\bar{T}_n}{n} - \frac{n^{2}}{6}{u_c^{*}}^3}{n^{3/2}}.
\end{equation*}
By \eqref{speed_mf2} and Slutsky theorem, 
it is enough to study the convergence in distribution of the variable $6(\frac{\bar{T}_n}{n} - \frac{n^{2}}{6}{u_c^*}^3)/n^{3/2}$.
We show that, for any $t\in\R$,
\be\label{goal}
\bar{M}_n (t):=\bar{\mathbb{E}}_{n;\alpha_c,h_c}\Big(e^{t \cdot	\bar{U}_n
}\Big) \, \xrightarrow{\;\; n \to +\infty \;\;} \,
\int_{\mathbb{R}}e^{ty}\bar \ell^c(y)dy,
\ee
where $\bar \ell^c$ is given in the statement. Again, we split the average in \eqref{goal} in two parts, one over $B_{u_c^*}$ given in \eqref{Bc}, and the other over $\cC \equiv \mathcal{C}(n,\d) := \G_n \setminus \left(B_{u_c^*}\right)$. We obtain:
\begin{align}
\sum_{m\in \cC}
\frac{\cN_m e^{t\sqrt{n}(m^3-{u_c^{*}}^3) +n^{2}(\frac{\alpha_c}{6}m^{3} + \frac{h_c}{2}m)}}
{\bar{Z}_{n;\a_c,h_c}}
&\overset{\eqref{Stirling-rough}}{\leq} \sum_{m\in \cC}
\frac{ C e^{3t\sqrt{n} +n^{2}g_{\alpha_c,h_c}(m)}}
{\bar{Z}_{n;\a_c,h_c}} \notag\\
&\leq 
\frac{C n^3 e^{3t\sqrt{n} +n^{2}f_{\a_c,h_c}-n^2 (g_{\a,h}(u^*_c)- \max_{m\in\cC} g_{\a_c,h_c}(m))}}{\bar{Z}_{n;\a_c,h_c}}\notag\\
&\leq c^{-1}C n^3 e^{3t\sqrt{n}-kn^{2-4\delta}} \xrightarrow{\;\; n \to +\infty \;\;} 0, \label{lim0}
\end{align}
for some constant $k>0$.
In the second to last inequality we used that $f_{\a_c,h_c}=g_{\a_c,h_c}(u^*_c)$, and that the set $\mathcal{C}$ contains at most $\binom{n}{2}$ elements. In the last inequality we used the Taylor expansion \eqref{taylor_g} at the critical point, the  rough bound $\bar{Z}_{n;\a,h} \geq cn^{-1} e^{n^2 f_{\a,h}}$ coming from \eqref{Stirling-rough}, and the fact that \mbox{$|m-u_c^*|^4>n^{-4\d}$} for $m\in\cC$. The assumption $\d<3/8$ guarantees that $2-4\d>1/2$.
 As a consequence of \eqref{lim0}, we can reduce the average in \eqref{goal} into a sum on $B_{u_c^*}$:
\be\label{base}
\begin{split}
	\bar{M}_n (t) &= \sum_{m\in B_{u_c^*}} 
	\frac{\cN_m e^{ t\sqrt{n}(m^3-{u_c^*}^3)+ n^{2}(\frac{\alpha_c}{6}m^{3} + \frac{h_c}{2}m)}}
	{\bar{Z}_{n;\a_c,h_c}}(1+o(1))\\
	&= \sum_{m\in B_{u_c^*}}\frac 2 {n^{\frac 32}}\frac{e^{t\sqrt{n}(m^3-{u_c^{*}}^3)-n^2 (f_{\a_c,h_c}-g_{\a_c,h_c}(m))}}
	{\sqrt{m(1-m)} \cdot D_c^{(n)}}(1+o(1))\,,
\end{split}
\ee
where the last identity is due to Lem.~\ref{lemma_Z} and the Stirling approximation \eqref{Stirling}.
Injecting in \eqref{base} the Taylor expansion \eqref{taylor_g} at the critical point, we get
\begin{equation}
	\begin{split}
		\bar{M}_n (t)&=
		\sum_{m\in B_{u_c^*}} \frac 2 {n^{\frac 32}}
		\frac{e^{t\sqrt{n}(m^3-{u_c^{*}}^3)-\frac{81}{64}n^{2}(m-u_c^{*})^{4}+ k_c n^{2}(m-u_c^{*})^{5}}}
		{\sqrt{m(1-m)} \cdot D_c^{(n)}}(1+o(1)).
	\end{split}
\end{equation}
By the change of variable $y=\sqrt{n}(m-u_c^{*})$, and recalling that
$u_c^{*}(\alpha_c,h_{c})=\frac{2}{3}$, we find
\begin{equation}
	\begin{split}
		\bar{M}_n (t)&=
		\sum_{y\in R_{c,\delta}^{(n)}} \frac 2 {n^{\frac 32}}
		\frac{e^{t(\frac 4 3y +2\frac{y^2}{\sqrt{n}}+\frac{y^3}{n})} \cdot e^{ -\frac{81}{64}y^{4} +  k_c \frac{y^5}{\sqrt{n}}}}
		{ \sqrt{\left(u_c^*+\frac{y}{\sqrt{n}}\right) \left(1-u_c^* -\frac{y}{\sqrt{n}}\right)} D_c^{(n)}}
		(1+o(1)).
	\end{split}
\end{equation}
Exploiting the range of $R_{c,\delta}^{(n)}$, given in \eqref{Rc}, we observe that 
$- n^{1/2-3\delta} <2\frac{y^2}{\sqrt{n}}+\frac{y^3}{n}<2n^{1/2-2\delta} + n^{1/2-3\delta}$. By isolating the term 
\begin{equation}\label{Mstar}
\bar{M}^{*}_n (t):=\sum_{y\in R_{c,\delta}^{(n)}} \frac 2 {n^{\frac 32}}
		\frac{e^{t\frac 4 3y } \cdot e^{ -\frac{81}{64}y^{4} +  k_c \frac{y^5}{\sqrt{n}}}}{ \sqrt{\left(u_c^*+\frac{y}{\sqrt{n}}\right) \left(1-u_c^* -\frac{y}{\sqrt{n}}\right)} D_c^{(n)}}
		(1+o(1))
\end{equation}
 we obtain:
\begin{equation}\label{sandw0}
e^{-tn^{1/2-3\delta}}\bar{M}^{*}_n (t)\leq \bar{M}_n (t) \leq 
 e^{t\left(2n^{1/2-2\delta} + n^{1/2-3\delta}\right)}\bar{M}^{*}_n (t).
\end{equation}
In \eqref{Mstar} we recognize the probability density
$\ell_n^c$ of the random variable $Y_n$, introduced in \eqref{densc}. From  \eqref{Mstar} we then deduce that $\bar{M}^{*}_n (t)= \E(e^{t\frac43Y_n})(1+o(1))\,$. By Scheff\'e  Lemma $Y_n$ converges in distribution to a generalized Gaussian random variable $Y$ with Lebesgue density \mbox{$\ell^c(y)\propto e^{-\frac{81}{64}y^4}$}, therefore $$\bar{M}^{*}_n (t)= \E(e^{t\frac43Y_n})(1+o(1))\xrightarrow{\;\; n \to +\infty \;\;} \E(e^{t\frac43Y}).$$ 
With the further constraint  $\frac{1}{4}<\delta<\frac{3}{8}$, from \eqref{sandw0} it holds $\bar{M}_n (t)\xrightarrow{n \to +\infty} \E(e^{t\frac43Y}).$ By setting $\bar Y:=\frac 4 3 Y$ we conclude the proof.
\endproof
\begin{proof}[Proof of Theorem~\ref{Thm_standard_CLT}]
The proof runs exactly as for the critical case, therefore we provide below only a sketch with the main differences. The object that we want to study is in this case the random variable
	\begin{equation}\label{decomp_cltst}
		\sqrt{6} \, \frac{\frac{\bar{T}_n}{n} - \frac{n^{2}}{6}\bar m^{\Delta}_n(\a,h)}{n}= \bar{V}_n+\frac{n}{\sqrt{6}}\big({u^*}^3-\bar m^{\Delta}_n(\a,h)\big),
	\end{equation}
	where
	\begin{equation*}
		\bar V_n:=\sqrt{6} \, \frac{\frac{\bar{T}_n}{n} - \frac{n^{2}}{6} {u_0^*}^3}{n}.
	\end{equation*}
By \eqref{speed_mf1} and Slutsky theorem, it is enough to study the moment generating function of of $\bar V_n$, restricting again the analysis on the neighborhood $B_{u^{*}_0}$ (the contribution over the set $\cC \equiv \mathcal{C}(n,\d) := \G_n \setminus B_{u_0^*}$ can be treated as in \eqref{lim0}, with $0<\delta<1$). To simplify constants, we consider $\sqrt{6}\bar V_n$ instead of $\bar V_n$. We then get:
	\begin{equation}
		\begin{split}
			\bar{M}_n (t)
			&=
			\sum_{m\in B_{u_0^*}} \frac 2 n\frac{1}{\sqrt{m(1-m)}} 
			\frac{e^{t{n}(m^3-{u_0^{*}}^3)-c_0n^{2}(m-u_0^{*})^{2}+ k_0 n^{2}(m-u_0^{*})^{3}}}
			{D_0^{(n)}}(1+o(1)).\\
		\end{split}
	\end{equation}
The change of variable \mbox{$x={n}(m-{u_0^{*}})$}, identity \eqref{cubic-expansion}, and the Taylor expansion \eqref{taylor_g} yield
	\begin{equation}
		\begin{split}
			\bar{M}_n (t)
			&=
			\sum_{x\in R_{0,\delta}^{(n)}} \frac 2 n
			\frac{e^{t(3{u_0^*}^2 {x} +3 {u_0^*}\frac{x^2}{n}+\frac{x^3}{n^{2}})} \cdot e^{ -c_0 x^{2} +  k_0 \frac{x^3}{n}}}
			{ \sqrt{\left(u_0^*+\frac{x}{{n}}\right) \left(1-u_0^* -\frac{x}{{n}}\right)} D_0^{(n)}}
			(1+o(1))\,.
		\end{split}
	\end{equation}
As in the proof of Thm.~\ref{Thm_non-standard_CLT}, let 
\begin{equation}\label{Mstar2}
\bar{M}^{**}_n (t):=\sum_{x\in R_{0,\delta}^{(n)}} \frac 2 n
			\frac{e^{3t{u_0^*}^2 x} \cdot e^{ -c_0 x^{2} +  k_0 \frac{x^3}{n}}}
			{ \sqrt{\left(u_0^*+\frac{x}{{n}}\right) \left(1-u_0^* -\frac{x}{{n}}\right)} D_0^{(n)}}(1+o(1)).
\end{equation}
Furthermore, by exploiting the range of $R_{0,\delta}^{(n)}$, given in \eqref{Ri} we obtain
\begin{equation}\label{sandw}
e^{-tn^{1-3\delta}}\bar{M}^{**}_n (t)\leq \bar{M}_n (t) \leq 
 e^{t\left(3u_0^* n^{1-2\delta} + n^{1-3\delta}\right)}\bar{M}^{**}_n (t).
\end{equation}
In \eqref{Mstar2} we recognize the probability density
$\ell_n$ of the random variable $X_n$, introduced in \eqref{dens_nc}; we then rewrite \eqref{Mstar2} as $\bar{M}^{**}_n (t)= \E(e^{3t {u_0^{*}}^2 X_n})(1+o(1))$.
By Scheff\'e Lemma $X_n$ converges in distribution to a real random variable $X$ with Gaussian density $\ell$ given in \eqref{liml}, therefore $$\bar{M}^{**}_n (t)= \E(e^{3t {u_0^{*}}^2 X_n})(1+o(1))\xrightarrow{\;\; n \to +\infty \;\;} \E(e^{3t {u_0^{*}}^2 X}).$$ 
With the further constraint  $\frac{1}{2}<\delta<1$, from \eqref{sandw}, it holds $\bar{M}_n (t)\xrightarrow{n \to +\infty} \E(e^{3t {u_0^{*}}^2 X})$.
Note that $X$ is a centered Gaussian random variable with variance  $(2c_0)^{-1}$, where
$c_0\equiv c_0(\a,h)=\frac{1-2\alpha[u_0^*(\alpha,h)]^2[1-u_0^*(\alpha,h)]}{4u_0^*(\alpha,h)[1-u_0^*(\alpha,h)]}.$
In conclusion, $\bar V_n$ converges in distribution to the centered Gaussian random variable $3{u_0^{*}}^2 X/\sqrt{6}$, with variance  $\bar{v}_0^{\Delta}(\a,h)=\frac{3{u^*_0}^4}{4c_0}$, as wanted.
\end{proof}
\paragraph{Generalization to a clique graph.} 
\begin{proof}[Proof of Thm.~\ref{Thm_LLN_mf_g}]
From Thm.~\ref{Thm_LLN_mf} we know that
$\left(\frac{2E_n}{n^2}\right)^3 \overset{\text{a.s.}}{\longrightarrow} {u_0^*}^3$ w.r.t.~$\bar{\mathbb{P}}_{\alpha,h}$. By continuous mapping theorem it follows that $\left(\frac{2E_n}{n^2}\right)^{\binom{\ell}{2}} \overset{\text{a.s.}}{\longrightarrow} {u_0^*}^{\binom{\ell}{2}}$ w.r.t.~$\bar{\mathbb{P}}_{\alpha,h}$, thus recovering \eqref{SLLN_gen}.\\
In order to prove \eqref{conv_distr_gen} one can proceed as in Thm.~\ref{Thm_convergence_in_distribution_mf}. The final goal is to determine the limit of $ \mathbb{\bar{E}}_{n;\alpha,h} \left[ \varphi \left(\frac{ \aut\bar{K}_n}{n^{\ell}} \right)\right]$ for any continuous and bounded real function $\varphi$. Since $m\equiv m(x)=\frac{2 E_n(x)}{n^2}$, from the definition of $\bar{K}_n$  we obtain $\bar{K}_n(x)= \frac{n^{\ell} m^{\binom{\ell}{2}}}{\aut}$. Using \eqref{mf-hamiltonian-scalar}, we get:
\begin{equation}\label{media_gen}
	\bar{\E}_{n;\a,h}\left[\varphi \left(\frac{ \aut\bar{K}_n}{n^{\ell}} \right) \right]
	= \sum_{m\in\G_n} \varphi(m^{\binom{\ell}{2}}) \, \frac{\cN_m e^{n^2 (\frac{\alpha}{6}m^{3} + \frac{h}{2}m)}}{\bar{Z}_{n;\alpha,h}}.
\end{equation}
One can then retrace exactly the same steps of the proof of Thm.~\ref{Thm_convergence_in_distribution_mf}, since the unique difference consists in the factor $\varphi(m^{\binom{\ell}{2}})$. Indeed, the measure (second factor in \eqref{media_gen}) remains unchanged. 
\end{proof}
\begin{proof}[Proof of Thm.~\ref{generalized_Thm_standard_CLT}]
We consider the decompositions
\begin{align*}
\sqrt{\aut} \, \frac{\frac{\bar{K}_n}{n^{\ell-2}} - \frac{n^{2}}{\aut}\bar m_n^{K}(\a,h)}{n}&= \bar{V}^{(\ell)}_n+\frac{n}{\sqrt{\aut}}\big({u^{*}_0}^{\binom{\ell}{2}}-\bar m^{K}_n(\a,h)\big), \\
\aut \, \frac{\frac{\bar{K}_n}{n^{\ell-2}} - \frac{n^{2}}{\aut}\bar m^{K}_n(\a_c,h_c)}{n^{3/2}}
&= \bar U^{(\ell)}_n
+ \sqrt{n}\left({u_c^{*}}^{\binom{\ell}{2}} - \bar m^{K}_n(\a_c,h_c)\right),
\end{align*}
where $\bar{V}^{(\ell)}_n:=\sqrt{\aut} \, \frac{\frac{\bar{K}_n}{n^{\ell-2}} - \frac{n^{2}}{\aut} {u_0^*}^{\binom{\ell}{2}}}{n}$ and $\bar{U}^{(\ell)}_n:= \aut \, \frac{\frac{\bar{K}_n}{n^{\ell-2}} - \frac{n^{2}}{\aut}{u_c^{*}}^{\binom{\ell}{2}}}{n^{3/2}}$. One can easily check that Cor.~\ref{cor_speed_mf}  can be adapted to this setting. It is enough to replace identities \eqref{cubic-expansion}--\eqref{cubic-expansion_c} respectively by
$\left(u_0^{*} +\frac{x}{n}\right)^{\binom{\ell}{2}} - {u_0^{*}}^{\binom{\ell}{2}}= \binom{\ell}{2}\left(u_0^*\right)^{\binom{\ell}{2}-1}\frac xn (1+o(1))$, where $x:=n(m-{u_0^*})$,  and $\big(u_c^{*} +\frac{y}{\sqrt{n}}\big)^{\binom{\ell}{2}} - {u_c^{*}}^{\binom{\ell}{2}}= \binom{\ell}{2}\left(\frac{2}{3}\right)^{\binom{\ell}{2}-1}\frac{y}{\sqrt{n}} (1+o(1))$, where $y:=\sqrt{n}(m-{u_c^*})$. Thanks to Cor.~\ref{cor_speed_mf} and Slutsky theorem we can reduce to study the convergence in distribution of the random variables $\bar{V}^{(\ell)}_n$ and $\bar{U}^{(\ell)}_n$. Adapting the proof of Thm.~\ref{Thm_standard_CLT} one can conclude that 
\begin{equation*}
\begin{aligned}
\bar{V}^{(\ell)}_n &\xrightarrow{\;\;\mathrm{d}\;\;} \bar Q:=\frac{\binom{\ell}{2}(u^*_0)^{\binom{\ell}{2}-1}}{\sqrt{\aut}}X &&\quad\text{w.r.t. } \bar{\mathbb{P}}_{n;\alpha,h},\text{ as } n \to +\infty, \\
\bar{U}^{(\ell)}_n &\xrightarrow{\;\;\mathrm{d}\;\;} \bar W:=\binom{\ell}{2}\left(\frac{2}{3}\right)^{\binom{\ell}{2}-1} Y &&\quad\text{w.r.t. } \bar{\mathbb{P}}_{n;\alpha_c,h_c},\text{ as } n \to +\infty, 
\end{aligned}
\end{equation*}
where $X$ is a centered Gaussian random variable with variance  $(2c_0)^{-1}$, being $c_0$ defined in Thm.~\ref{Thm_standard_CLT}, and $Y$ has Lebesgue density $\ell^c(y)\propto e^{-\frac{81}{64}y^4}$. Hence $\bar{Q}$
has variance  $\bar v^{K}(\alpha,h) := \nicefrac{\Big(\tbinom{\ell}{2}{u_0^*}^{\binom \ell2-1}\Big)^2}{2\aut c_0}$, and $\bar{W}$ has scale parameter $\frac{2^{\binom{\ell}{2} + 1/2}}{3^{\binom{\ell}{2}}}\binom{\ell}{2}$, as wanted.
\end{proof}
\subsection{Conditional measures}

\proof[Proof of Proposition \ref{prop_speed_conditional_mfm}] 
Let $(\alpha,h) \in \mathcal{M}^{rs}$ and let $u^*_i=u^*_i(\a,h)$, $i=1,2$ the two solutions of the scalar problem \eqref{free_energy}.  The proof of this proposition can be carried on exactly as the proof of the analog Prop.~\ref{prop_speed_mmf}, but in the conditional setting introduced in Subsec.~\ref{condm}. Without loss of generality, we consider the case $i=1$: 
\begin{align*}
	\hat{\mathbb{E}}_{n;\a,h}^{(1)}\left( \left\vert \frac{6\bar{T}_n}{n^3} - {u_1^*}^3 \right\vert \right) 
	&= \sum_{m\in B_{u^*_1}} |m^3-{u_1^*}^3| \frac{\cN_m \, e^{n^{2}(\frac{\alpha}{6}m^{3} + \frac{h}{2}m)}} {\bar{Z}_{n;\a,h}(B_{u_1^*})}\\
	&= \sum_{m\in B_{u^*_1}} \frac 2 n\frac{|m^3-{u_1^*}^3|}{\sqrt{m(1-m)}} \frac{e^{-n^2 (f_{\a,h}-g_{\a,h}(m))}}
	{D_1^{(n)}}(1+o(1)),
\end{align*}
where $\hat{\mathbb{E}}_{n;\a,h}^{(1)}$ is the expectation associated with the measure $\hat{\mathbb{P}}_{n;\alpha,h}^{(1)}$ defined in \eqref{conditional_measure}.
The Taylor expansion \eqref{taylor_g} and the change of variable $x= n(m-{u^*_1})$ yield
\[
n \cdot \hat{\mathbb{E}}^{(1)}_{n;\a,h}\left( \left\vert\frac{6\bar{T}_n}{n^3}-{u_1^*}^3\right\vert\right) =
\sum_{x\in R^{(n)}_{1,\delta}}  \frac 2 n
\frac{| 3{u_1^*}^2 {x} +3 {u_1^*}\frac{x^2}{n}+\frac{x^3}{n^{2}}|
	\cdot  e^{-c_1x^2 + \frac{k_1}{n}x^3 }}
{\sqrt{\left(u_1^*+\frac{x}{{n}}\right) \left(1-u_1^* -\frac{x}{{n}}\right)}\cdot D_1^{(n)}}(1+o(1)),
\]
where $R^{(n)}_{1,\delta}$ 
is as in \eqref{Ri}. If $(X^{(1)}_n)_{n \geq 1}$ is a sequence of random variables with probability density
\be \label{ell_1}
\ell_n^{(1)}(x):= \frac 2 n 
\frac{e^{-c_1x^2 + \frac{k_1}{n} x^3 }}
{{\sqrt{\left(u_1^*+\frac{x}{{n}}\right) \left(1-u_1^* -\frac{x}{{n}}\right)}}\cdot D_1^{(n)}} \, \1_{R^{(n)}_{1,\delta}} (x),\quad x\in\R,
\ee
where $D_1^{(n)}$ is the normalization weight defined in \eqref{Di}, we obtain, as in \eqref{final_sum_a}--\eqref{final_sum_b},
\begin{align}
&\E(|3{u_1^*}^2X^{(1)}_n|)(1+o(1)) - \left(\int_{\mathbb{R}}\left|3 {u_1^*}\frac{x^2}{n}+\frac{x^3}{n^{2}}\right|\ell^{(1)}_{n}(x) dx\right) (1+o(1)) \label{final_sumc_a} \\
&\hspace{3cm  } \leq n \cdot \hat{\E}^{(1)}_{n;\a,h}
			\Big( \Big\vert \frac{6\bar{T}_n}{n^3} - {u_1^*}^3 \Big\vert\Big) \leq  \notag \\ 
&  \E(|3{u_1^*}^2X^{(1)}_n|)(1+o(1)) + \left(\int_{\mathbb{R}}\left|3 {u_1^*}\frac{x^2}{n}+\frac{x^3}{n^{2}}\right|\ell^{(1)}_{n}(x) dx\right) (1+o(1)). \label{final_sumc_b}
\end{align}

Arguing as at the end of proof \ref{prop_speed_mmf} we conclude
$$n\cdot\hat{\E}^{1}_{n;\a,h}\Big( \Big\vert \frac{6\bar{T}_n}{n^3} - {u_1^*}^3 \Big\vert\Big) \xrightarrow{\;\; n \to +\infty \;\;} \, 3{u_1^*}^2\E(|X^{(1)}|),$$
where $X^{(1)}$ is a standard Gaussian variable with variance $(2c_1)^{-1}$. Indeed, the second summand in \eqref{final_sumc_a}--\eqref{final_sumc_b} vanishes, being $3 {u_1^*}\frac{x^2}{n}+\frac{x^3}{n^{2}}=o(1)$, for fixed $x$.
Setting $\bar{X}^{(1)}:= 3{u_1^*}^2 X^{(1)}$, we obtain a random variable with variance $6\bar{v}_1^{\Delta}(\alpha,h) = \frac{9{u_i^*(\a,h)}^4}{2c_1}$, as wanted.
The same proof holds for the case $i=2$.
\endproof
\proof[Proof of Corollary \ref{cor_speed_conditional_mfm}]
The proof follows immediately, as for Cor.~\ref{cor_speed_mf}. 
\endproof
As mentioned, the next theorem is the analog of Thm.~\ref{Thm_LLN_mf} and Thm.~\ref{Thm_standard_CLT}, when the edge density is conditioned to take values in a neighborhood of the two maximizers of the scalar problem \eqref{free_energy}.

\proof[Proof of Theorem \ref{Thm_conditional_limit_theorems_mfm}]
Let $(\alpha,h) \in \mathcal{M}^{rs}$ and let $u^*_i=u^*_i(\a,h)$, $i=1,2$ the two solutions of the scalar problem \eqref{free_energy}. We focus on the case $i=1$, being the case $i=2$ completely analogous. We start proving \eqref{conditional_lln_mfm} via exponential convergence, which again implies the \emph{a.s.}~convergence by a standard Borel-Cantelli  argument (see~\cite{E}, Thm.~II.6.4 and Rem.~\ref{consist_cond}).
We define 
$$\mathcal{R} \equiv \mathcal{R}(\eta;n) := \left\{ m \in \Gamma_n: \eta \leq \left\vert m^3 - {u_1^*}^3 \right\vert < n^{-\delta} \right\}.$$ 
This is a circular crown of points with distance to ${u^*_1}^3$, between $\eta$ and $n^{-\delta}$. Notice that for fixed $\eta$, for large $n$, the set  $\mathcal{R}$ is empty. When this does not hold, we have
\begin{align}
	\hat{\mathbb{P}}^{(1)}_{n;\alpha,h} \left( \left\vert \frac{6\bar{T}_n}{n^3} - {u_1^*}^3\right\vert \geq \eta \right) &\leq \sum_{m \in \mathcal{R}}  \frac{\mathcal{N}_m e^{n^2 \left(\frac{\alpha}{6}m^3+\frac{h}{2}m\right)}}{\bar{Z}_{n;\alpha,h}(B_{u_1^*})} \notag\\
	&\leq  C c^{-1} n^4 e^{-n^2 (f_{\alpha,h}-\max_{m \in \mathcal{R}} g_{\alpha,h}(m))} \notag\\
	&\leq C c^{-1} n^4 e^{-n^2 \min_{m \in \mathcal{R}} (f_{\alpha,h} - g_{\alpha,h}(m))}, \label{exprob}
\end{align}
where in the second to last passage we used the rough bound $\bar{Z}_{n;\alpha,h}(B_{u_1^*}) \geq cn^{-1} e^{n^2 f_{\alpha,h}}$ coming from Stirling approximation \eqref{Stirling-rough}.
As stated in~\cite[Prop.~3.2]{RY}, for sufficiently large $n$,  the function $f_{\alpha,h} - g_{\alpha,h}(m)$, restricted to the neighborhood  $B_{u_1^*}$, is positive, convex and admits $u_1^*$ as unique zero. Hence 
\[
\min_{m \in \mathcal{R}} (f_{\alpha,h} - g_{\alpha,h}(m)) = \min\{f_{\alpha,h} - g_{\alpha,h}(u_1^*-\eta), f_{\alpha,h} - g_{\alpha,h}(u_1^*+\eta)\}>0,
\]
When $\mathcal{R}$ is nonempty, the probability in \eqref{exprob} vanishes, as $n\to\infty$. This provides the desired exponential convergence, for every choice of $\eta>0$.
We now move to the proof of \eqref{conditional_clt_mfm}.  By means of decomposition \eqref{decomp_cltst}, and Cor.~\ref{cor_speed_conditional_mfm}  we can reduce our analysis to the random variable 
	\begin{equation*}
		\bar V^{(1)}_n:=\sqrt{6} \, \frac{\frac{\bar{T}_n}{n} - \frac{n^{2}}{6} {u_1^*}^3}{n},
	\end{equation*}
studying, for any $t\in\R$, its moment generating function
\[
\hat{M}_n(t) := \hat{\mathbb{E}}^{(1)}_{n;\alpha,h}\left(e^{t
	\bar V^{(1)}_n
}\right).
\]
We consider $\sqrt{6}\bar V^{(1)}_n$ instead of $\bar{V}^{(1)}_n$ (to simplify constants), and we follow the same line of arguments as in the proof of Thm.~\ref{Thm_standard_CLT}.
We  get:
	\begin{equation}
		\begin{split}
			\hat{M}_n (t)
			&=
			\sum_{m\in B_{u_1^*}} \frac 2 n\frac{1}{\sqrt{m(1-m)}} 
			\frac{e^{t{n}(m^3-{u_1^{*}}^3)-c_1n^{2}(m-u_1^{*})^{2}+ k_1 n^{2}(m-u_1^{*})^{3}}}
			{D_1^{(n)}}(1+o(1)).\\
		\end{split}
	\end{equation}
The change of variable \mbox{$x={n}(m-{u_1^{*}})$}, identity \eqref{cubic-expansion}, and the Taylor expansion \eqref{taylor_g} yield
	\begin{equation}
		\begin{split}
			\hat{M}_n (t)
			&=
			\sum_{x\in R_{1,\delta}^{(n)}} \frac 2 n
			\frac{e^{t(3{u_1^*}^2 {x} +3 {u_1^*}\frac{x^2}{n}+\frac{x^3}{n^{2}})} \cdot e^{ -c_1 x^{2} +  k_1 \frac{x^3}{n}}}
			{ \sqrt{\left(u_1^*+\frac{x}{{n}}\right) \left(1-u_1^* -\frac{x}{{n}}\right)} D_1^{(n)}}
			(1+o(1))\, ,
		\end{split}
	\end{equation}
where $R_{1,\delta}^{(n)}$ is defined in \eqref{Ri}. By defining 

\begin{equation}\label{Mstar2_hat}
\hat{M}^{**}_n (t):=\sum_{x\in R_{1,\delta}^{(n)}} \frac 2 n
			\frac{e^{3t{u_1^*}^2 x} \cdot e^{ -c_1 x^{2} +  k_1 \frac{x^3}{n}}}
			{ \sqrt{\left(u_1^*+\frac{x}{{n}}\right) \left(1-u_1^* -\frac{x}{{n}}\right)} D_1^{(n)}}(1+o(1)),
\end{equation}
we observe that  $\hat{M}^{**}_n (t)= \E(e^{3t {u_1^*}^2 {X}^{(1)}_n})(1+o(1))\,$, where, for each $n\in\mathbb{N}$, ${X}^{(1)}_n$ is a random variable with density $\ell^{(1)}_n(x)$ given in \eqref{ell_1}.

Notice that ${X}^{(1)}_n$ converges in distribution to a centered Gaussian random variable ${X}^{(1)}$ with variance  $(2c_1)^{-1}$, where
$c_1\equiv c_1(\a,h)=\frac{1-2\alpha[u_1^*(\alpha,h)]^2[1-u_1^*(\alpha,h)]}{4u_1^*(\alpha,h)[1-u_1^*(\alpha,h)]}.$

Therefore 
$$\bar{M}^{**}_n (t)= \E(e^{3t {u_1^*}^2 {X}^{(1)}_n})(1+o(1))\xrightarrow{\;\; n \to +\infty \;\;} \E(e^{3t {u_1^*}^2 {X}^{(1)}}).$$ 
With the further constraint  $\frac{1}{2}<\delta<1$, exploiting the same bounds as in \eqref{sandw}, we also obtain the convergence $\hat{M}_n (t)\xrightarrow{n \to +\infty} \E(e^{3t {u_1^*}^2 {X}^{(1)}})$ for all $t\in\R$.
In conclusion, $\bar{V}^{(1)}_n$ converges in distribution to the centered Gaussian random variable $3{u_1^*}^2 {X}^{(1)}/\sqrt{6}$, with variance  $\bar{v}_{1}^{\Delta}(\a,h)=\frac{3{u^*_1}^4}{4c_1}$, as wanted. The same proof works for $i=2$. 
\endproof

\section{Conjectures and simulations} \label{CLT_discussion}
The comparison between the mean-field approximation and the edge-triangle model respectively encoded by Hamiltonian \eqref{Hamilt_ERG} and \eqref{Hamilt_MF} remains an open problem. We refer the reader to \cite[Sec.~8.3]{BCM} for a discussion on the main difficulties in proving that they asymptotically coincide. 
However we believe that this is the case, and we report a list of conjectures based on the results obtained for the mean-field model. Some of them are supported by heuristics computations and simulations. 

\subsection{Conjectures and heuristics}
\begin{conjecture}\label{cong3}
	For all $(\alpha,h) \in \mathcal{M}^{rs}$,
	\[
	\frac{6 T_n}{n^3} \, \xrightarrow{\;\;\mathrm{d}\;\;}{} \, \kappa \delta_{{u_1^{*}}^3(\alpha,h)}+
	(1-\kappa)\delta_{{u_2^{*}}^3(\alpha,h)}  \quad \text{ w.r.t. } {\mathbb{P}}_{n;\alpha,h}, \text{ as } n \to +\infty,
	\]
	where $u_1^*$, $u_2^*$ solve the maximization problem in \eqref{free_energy}, and
	$$\kappa=\frac{
		\sqrt{\left[{1-2\alpha \left( u_1^*\right)^2(1-u_1^*)}\right]^{-1}}
	}
	{
		\sqrt{\left[{1-2\alpha \left( u_1^*\right)^2(1-u_1^*)}\right]^{-1}}
		+
		\sqrt{\left[{1-2\alpha \left( u_2^*\right)^2(1-u_2^*)}\right]^{-1}}
	} \,.
	$$
\end{conjecture}

	\begin{conjecture}[CLT for $T_n$]\label{cong1}
	If $(\a,h) \in \mathcal{U}^{rs}\setminus \{(\a_c,h_c)\}$, 
	\[	
	\sqrt{6} \, \frac{\frac{{T}_n}{n} - \frac{n^{2}}{6} m^{\Delta}_n(\a,h)}{n} \xrightarrow{\;\;\mathrm{d}\;\;}{}  \mathcal{N}(0, v_0^\Delta(\alpha,h)) \quad \text{ w.r.t. } {\mathbb{P}}_{n;\alpha,h},\text{ as } n \to +\infty,
	\]
	where $\mathcal{N}(0, v_0^\Delta(\alpha,h))$ is a centered Gaussian distribution with
	variance 	
	$v_0^\Delta(\alpha,h)=  \frac{3{u^*_0}^4(\alpha,h)}{4 c_0}$,
being $c_0\equiv c_0(\alpha,h):= \frac{1-2\alpha[u^*_0(\alpha,h)]^2[1-u^*_0(\alpha,h)]}{4u^*_0(\alpha,h)[1-u^*_0(\alpha,h)]}$.
\end{conjecture}

\begin{conjecture}
	[Non-standard CLT for ${T}_n$]\label{cong2}
	If $(\a,h)=(\alpha_c,h_c)$,
	\begin{equation}\label{distrib_ns}	
	6 \, \frac{\frac{{T}_n}{n} - \frac{n^{2}}{6}m^{\Delta}_n(\a_c,h_c)}{n^{3/2}} \xrightarrow{\;\;\mathrm{d}\;\;}{}  Y \quad \text{ w.r.t. } {\mathbb{P}}_{n;\alpha_c,h_c},\text{ as } n \to +\infty,
	\end{equation}
	where $Y$ is a generalized Gaussian random variable with Lebesgue density ${\ell}^c(y)\propto e^{-\frac{3^8}{2^{14}}y^{4}}$.	
\end{conjecture}

Having at hand a large deviation principle (see Rem.~\ref{Rmk_tilted_LDP}) allows to perform a heuristic calculation to support Conjs.~\ref{cong1} and \ref{cong2}. 

\paragraph{Heuristics on the CLT.}
To guarantee convexity of the rate function $\mathcal{I}_{\alpha,h}$, we restrict here to the region $(\alpha,h)\in (-2,\alpha_c)\times \mathbb{R}$ (see \cite{RY}, Prop.~3.2). Let us define the random variables
		\begin{align}
			V_n&:= \sqrt{6} \, \frac{\frac{T_n}{n}-\frac{n^2}{6}{u_0^*}^3(\alpha,h)}{n}= \frac{n}{\sqrt{6}}\left[ \frac{6 T_{n}}{n^3} -{u_0^{*}}^3(\alpha,h)\right], \qquad (\alpha,h)\in (-2,\alpha_c)\times \mathbb{R}, \label{Vn}\\
			U_n &:= 6 \, \frac{\frac{T_n}{n}-\frac{n^2}{6}{u_0^*}^3(\alpha_c,h_c)}{n^{3/2}} = \sqrt{n}\left[ \frac{6 T_{n}}{n^3} -{u_0^{*}}^3(\alpha_c,h_c)\right],
		\end{align}
		so that we obtain the following decomposition:
		\begin{align}
			\sqrt{6} \, \frac{\frac{T_n}{n}-\frac{n^2}{6}m^{\Delta}_n(\alpha,h)}{n}&=
			V_n + n({u_0^*}^{3}(\alpha,h)-m^{\Delta}_n(\alpha,h)), \qquad (\alpha,h)\in (-2,\alpha_c)\times \mathbb{R} \label{decomp1}\\
			6 \, \frac{\frac{T_n}{n}-\frac{n^2}{6}m^{\Delta}_n(\alpha_c,h_c)}{n^{3/2}}&=
			U_n + \sqrt{n}({u_0^*}^{3}-m^{\Delta}_n(\alpha_c,h_c)). \label{decomp2} 
		\end{align}
For all $(\alpha,h)\in (-2,\alpha_c)\times \mathbb{R}$ let $C_{\epsilon} := \{y \in [0,1] : |y^3 - {u_0^*}^3(\alpha,h)| \geq \epsilon \}$ where we assume $0<  \epsilon < \min\{{u_0^{*}}^3(\alpha,h), 1- {u_0^{*}}^3(\alpha,h)\}$.  We claim:  
		\begin{align}\label{LDP_cong}
			\mathbb{P}_{n;\alpha,h}\left(\frac{6T_n}{n^3} \in C_{\epsilon} \right) \approx e^{-n^2\inf_{x\in C_{\epsilon}} I_{\alpha, h}(x) } 
		\end{align}
where $I_{\alpha,h}(x) = \frac{1}{2}I(x)  - \frac{\alpha}{6}x^3  - \frac{h}{2}x + f_{\alpha, h}$, and $I(x)=x\ln x + (1-x) \ln (1-x)$. 
Note that we have used the LDP of Rem.~\ref{Rmk_tilted_LDP},  exploiting the fact that whenever we work in replica symmetric regime, the variational problem $\inf_{\tilde{g}\in \widetilde{C}_{\epsilon}} \mathcal{I}_{\alpha, h}(\tilde{g})$ on the set $\widetilde{C}_{\epsilon}:=\{\tilde{g}\in \widetilde{W}: |t(H_2, \tilde{g})- {u_0^*}^3(\alpha,h)| \geq \epsilon \}$ reduces to the scalar problem on the r.h.s.~of \eqref{LDP_cong} (indeed $\mathcal{I}_{\alpha, h}$ coincides with $I_{\alpha,h}$ when it is computed on constant graphons).
 		\begin{figure}[h!]
			\centering
			\includegraphics[scale=0.7]{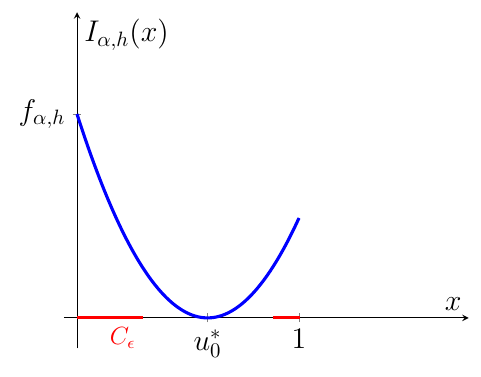}
\caption{\small Qualitative representation of $I_{\alpha,h}$ in $\mathcal{U}^{rs}$. The red interval corresponds to the set $C_{\epsilon}$ defined below \eqref{decomp2}.} \label{I}
		\end{figure}
	\noindent Moreover, since $(\alpha,h)\in (-2,\alpha_c)\times \mathbb{R}$, the function $I_{\alpha,h}$ is continuous, positive, and convex (see \cite{RY}, Prop.~3.2); furthermore, it admits a unique zero at $x = u_0^*$ \footnote{To the sake of readability in the following part of the text we omit the dependence on $(\alpha,h)$.} (see Fig.~\ref{I} for a qualitative representation $I_{\alpha,h}$). As a consequence,
		\begin{align} \label{Taylor}
		\inf_{x\in C_{\epsilon}} I_{\alpha, h}(x)&= I_{\alpha,h}\left(\sqrt[3]{{u_0^{*}}^{3} \pm \epsilon}\right)=I_{\alpha,h}\left(u_0^* \pm \frac{\epsilon}{3{u_0^{*}}^{2}} +o(\epsilon^2) \right), 
		\end{align}
where in last passage we used the assumption $0<  \epsilon < \min\{{u_0^{*}}^3(\alpha,h), 1- {u_0^{*}}^3(\alpha,h)\}$.  
Set $\delta:=\frac{\epsilon}{3{u_0^{*}}^2} +o(\epsilon^2)$; we determine which is the order of the functions $I_{\alpha,h}(u_0^{*}\pm \delta)$ as $\delta$ goes
to zero. We obtain:
		\begin{align}\label{rate_fct_expansion}
			I_{\alpha,h}(u_0^* \pm \delta) &= \pm \frac{\delta}{2} \left[ \ln \frac{u_0^*}{1-u_0^*} - \alpha (u_0^*)^2 - h\right] + \frac{\delta^2}{2} \left[ \frac{1}{2u_0^*(1-u_0^*)} - \alpha u_0^*\right] \\[.3cm]
			&\quad \pm \frac{\delta^3}{2} \left[ \frac{2u_0^*-1}{6(u_0^*)^2(1-u_0^*)^2} - \frac{\alpha}{3}\right] + \frac{\delta^4}{24} \, \frac{3(u_0^*)^2 -3u_0^*+1}{(u_0^*)^3(1-u_0^*)^3} + o(\delta^4).\nonumber
		\end{align}
		\begin{itemize}
			\item If $(\alpha,h)\in (-2,\alpha_c)\times \mathbb{R}$, then the first non-vanishing coefficient in \eqref{rate_fct_expansion} is the coefficient of the second order term, which equals $\frac{I_{\alpha,h}''(u_0^*)}{2}$, and it is strictly positive as $I_{\alpha,h}$ is strictly convex in the parameter range we are considering. Therefore, we get 
\begin{equation}\label{I_exp}			
I_{\alpha,h}(u_0^* \pm \delta) = \frac{I_{\alpha,h}''(u_0^*)}{2} \delta^2 + o(\delta^2),
\end{equation}
and, as $n\to\infty$, 	
\begin{align}
			\mathbb{P}_{n;\alpha,h} \left( V_n \in dx \right) = \mathbb{P}_{n;\alpha,h} \left( \tfrac{6T_n}{n^3} \in {u_0^*}^3+\frac{\sqrt{6}}{n} dx 
			\right) &
			\overset{\eqref{LDP_cong}\eqref{Taylor}}{\approx} e^{ -n^2 I_{\alpha,h} \left( \sqrt[3]{{u_0^*}^3 +\frac{\sqrt{6}}{n}x}\right)} dx \nonumber\\&= e^{-\frac{I''_{\alpha,h}(u_0^{*}) x^2}{3 {u_0^{*}}^4} + o({x^2})}dy, \label{dens1}
\end{align}
where in the last equality we used \eqref{I_exp} injecting $\epsilon= \frac{{\sqrt{6}x}}{n}$ in the definition of $\delta$.
In \eqref{dens1} we recognize the density of a Gaussian random variable with variance $\frac{3 {u_0^{*}}^4}{2I''_{\alpha,h}(u_0^{*})}$. A direct computation shows that $I''_{\alpha,h}(u_0^{*})=2c_0$ (where $c_0$ is given in \eqref{c}), hence we recover the value \eqref{varCLT} stated in Thm.~\ref{Thm_standard_CLT}.

\item If $(\alpha,h) = (\alpha_c,h_c) = \left( \frac{27}{8}, \ln 2 - \frac{3}{2} \right)$, since $u_0^*=u_0^*(\alpha_c,h_c) = \frac{2}{3}$, \eqref{rate_fct_expansion} reduces to
			$$I_{\alpha,h}(u_0^* \pm \delta) = \frac{81}{64} \, \delta^4 + o(\delta^4). $$
As $n \to +\infty$, we find
		\begin{align}
			\mathbb{P}_{n;\alpha_c,h_c} \left( U_n \in dx \right) = \mathbb{P}_{n;\alpha_c,h_c} \left( \tfrac{6T_n}{n^3} \in {u_0^*}^3+\tfrac{dx}{\sqrt{n}} \right)
			&\overset{\eqref{LDP_cong}\eqref{Taylor}}{\approx} e^{ -n^2 I_{\alpha_c,h_c} \left( \sqrt[3]{{u_0^*}^3 +\frac{{x}}{\sqrt{n}}}\right)} dx  \nonumber\\
&= e^{-\frac{3^8}{2^{14} } \, {x^4} + o({x^4})}dx, \label{dens2}
		\end{align}
where in the last equality we used \eqref{I_exp} injecting $\epsilon= \frac{{x}}{\sqrt{n}}$ in the definition of $\delta$. In \eqref{dens2} we can immediately recognize  the same density $\bar\ell^c$ stated in Thm.~\ref{Thm_non-standard_CLT}.
		\end{itemize}
However, notice that the error terms appearing in \eqref{dens1}--\eqref{dens2}
		might be relevant,  as well as the \emph{shift terms} in \eqref{decomp1}--\eqref{decomp2}, similarly to what happens in \cite[Thm.~1.4(c)]{MX} for the two-star model (indeed we don't have the equivalent of Cor.~\ref{cor_speed_mf}, which is valid instead for the mean-field model). However, we believe that a subtle compensation among this two contributions produces the conjectured results.
In the next section, we show two simulations that support Conj.~\ref{cong1}.

\subsection{Simulations} \label{simulations}
We perform a discrete-time  Glauber  dynamics, namely an ergodic reversible
Markov chain on $\mathcal{A}_n$ with stationary distribution $\mathbb{P}_{n;\alpha,h}$.
A step of the Glauber  dynamics can be described as follows: 
\begin{itemize}
\item[1.] Uniformly sample $\ell \in \cE_n$, and let $x^{+}$ (resp.~$x^{-}$) be the adjacency matrix, with $x^{+}_{\ell}=1$ (resp.~$x^{-}_{\ell}=0$), that coincides with $x$ for all elements except for $x_{\ell}$. 
Let $\mathcal{W}_{\ell}:=\{\{i,j\}: i,j\in \cE_n,  i\sim j, \{i,j,\ell\}\in \mathcal{T}_n \Leftrightarrow x_{\ell}=1\}$ the set of two-stars insisting on the edge (or non-edge) $x_{\ell}$. Here, the symbol $\sim$ denotes that the two edges $i$ and $j$ are neighbors. 
\item[2.] Given the current state represented by $x \in \cA_n$,  the next state is obtained by performing the transition $x\mapsto x^{+}$ (resp.~$x\mapsto x^{-}$) with probability 
\begin{equation}\label{pn}
	p_n(x,\ell) := \frac{e^{\alpha \sum_{\{i,j\}\in \mathcal{W}_{\ell}} x_{i}x_{j} +h}}{1+e^{\alpha \sum_{\{i,j\}\in \mathcal{W}_{\ell}} x_{i}x_{j} +h}} \quad (\text{resp.~}1- p_n(x,\ell)\,)\,.
\end{equation}
\end{itemize}
The update probability \eqref{pn} is given in \cite[Lem.~3]{BBS} (or equivalently \cite{BHLN}, pag.~18). Moreover in \cite[Thm.~5]{BBS} it has been proved that the mixing time of the Glauber  dynamics is $\Theta(n^2\ln(n))$ whenever $(\alpha,h)\in \mathcal{U}^{rs}$. 

\begin{figure}[h!]
	\begin{subfigure}{.5\textwidth}
		\centering
		\includegraphics[width=.8\linewidth]{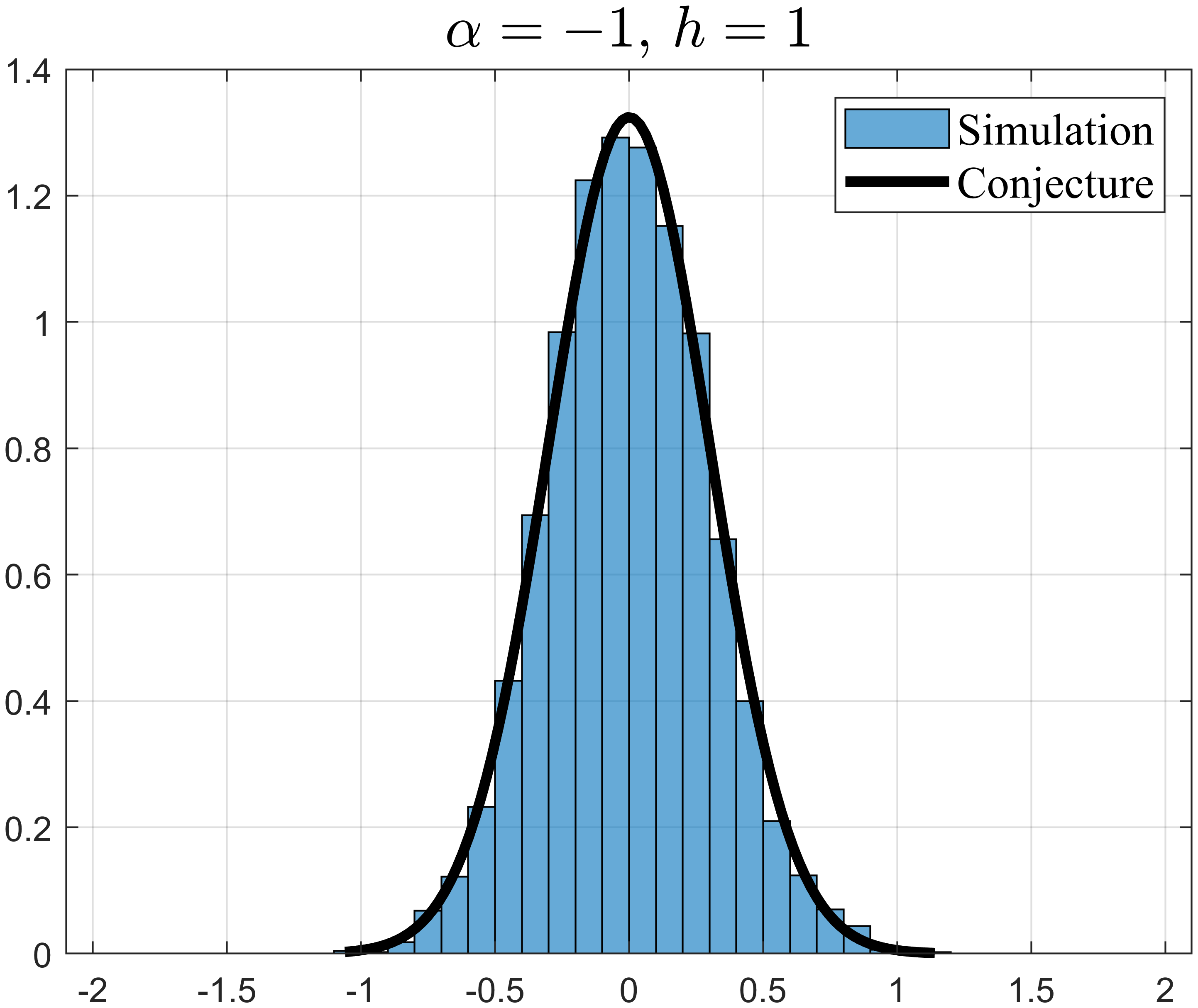}
	\end{subfigure}%
	\begin{subfigure}{.5\textwidth}
		\centering
		\includegraphics[width=.8\linewidth]{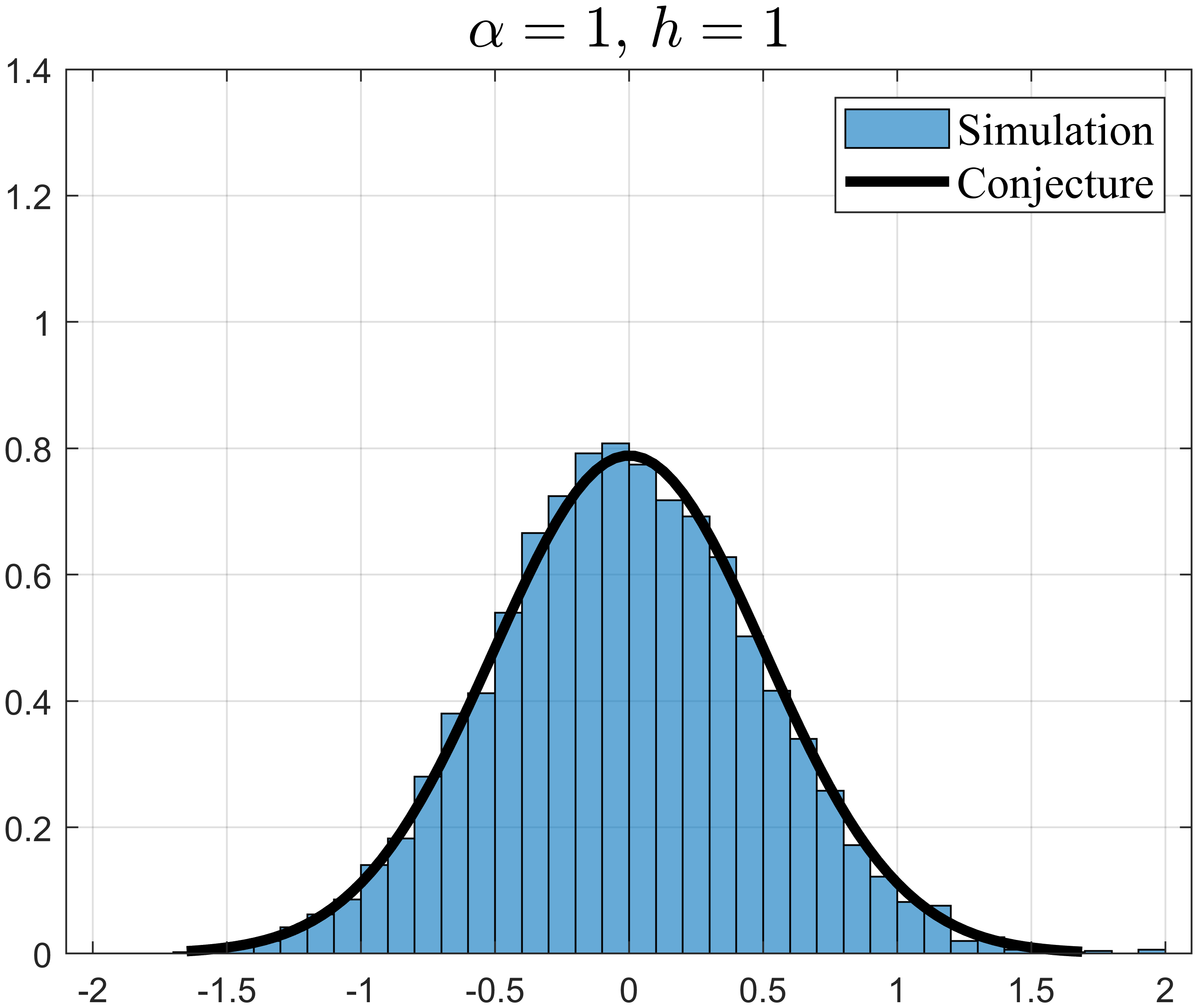}
	\end{subfigure}
	\caption{\small The picture displays a simulation of the distribution of $\sqrt{6} \, \frac{\frac{T_n}{n}-\frac{n^2}{6}m^{\Delta}_n(\alpha,h)}{n}$ obtained with $n=150$, $M=5000$ samples, and parameters $h = 1$, $\a=\pm1$ (histogram), and the pdf of the Gaussian distribution  introduced in Conj.~\ref{cong1} (continuous line).}
	\label{fig:clt-n150m5000h1}
\end{figure}

Pict.~\ref{fig:clt-n150m5000h1} shows a numerical simulation of the probability distribution of \eqref{distrib_ns} obtained with $n=150$ and $M=5000$ samples, both for negative and positive values of $\alpha$. The picture also displays the Gaussian probability density function given in Conj.~\ref{cong1}, showing that it approximates the histogram with good accuracy, thus supporting the conjecture.  
\begin{remark}
Note that, despite \cite[Thm.~5]{BBS} holds in $\mathcal{U}^{rs}$, which includes $(\alpha_c,h_c)$, when we perform the Glauber  dynamics at the critical point, the mixing time that we observe is not $\Theta(n^2\ln(n))$, as we would expect. We believe that the proximity of the point to the critical curve $\mathcal{M}^{rs}$ where the mixing time is exponential (see \cite[Thm.~6]{BBS}), is responsible for this behavior.  As a consequence, 
the incredibly high computational cost prevented us from getting an equivalent simulation for supporting Conj.~\ref{cong2} and Conj.~\ref{cong3}. 
\end{remark}

\noindent {\bf Acknowledgements.} We would like to thank Alessandra Bianchi and Francesca Collet for useful discussions and suggestions. \\
This research has been partially funded by the INdAM-GNAMPA Project, CUP E53C23001670001.

\end{document}